\documentclass[11pt,a4paper]{article}

\title{On GMRES for singular EP and GP systems}
\author{Keiichi Morikuni\thanks{Division of Information Engineering, Faculty of Engineering, Information and Systems, University of Tsukuba, Japan. Email: morikuni@cs.tsukuba.ac.jp, URL: http://researchmap.jp/KeiichiMorikuni/. The work was supported in part by JSPS KAKENHI Grant Number 16K17639.} \and Miroslav Rozlo\v{z}n\'{i}k\thanks{Institute of Mathematics, Czech Academy of Sciences, Prague, Czech Republic. Email: miro@math.cas.cz. The work was supported by the project GA17-12925S of the Czech Science Foundation.}}
\date{}

\usepackage{amsmath, amsfonts, amssymb, amsthm}
\usepackage[all, warning]{onlyamsmath} 
\usepackage{graphicx}
\usepackage[
	bookmarks=false,
	bookmarksnumbered=true,%
	bookmarksopen=false,	
	bookmarkstype=toc,%
	colorlinks={true},%
	pdfauthor={Keiichi Morikuni and Miroslav Rozloznik},%
	pdfdisplaydoctitle={true},
	pdftitle={On GMRES for singular EP and GP systems},%
	setpagesize={false},%
	urlcolor={red},
	pdfstartview={FitH -32768}
]{hyperref}

\hypersetup{pdfpagemode=UseNone} 

\usepackage{cleveref}
\crefformat{equation}{\textup{#2(#1)#3}}
\crefname{theorem}{Theorem}{Theorems}
\Crefname{figure}{Figure}{Figures}
\crefname{figure}{Figure}{Figures}
\crefname{proposition}{Proposition}{Propositions}

\usepackage{enumitem}
\usepackage[T1]{fontenc}
\usepackage{exscale, fix-cm, lmodern, textcomp}
\usepackage[top=1in, bottom=1in, left=1in, right=1in]{geometry}
\usepackage{algorithm, algorithmic}
\usepackage{subfig}

\hypersetup{pdfpagemode=UseNone}

\theoremstyle{plain}
\newtheorem{theorem}{Theorem}[section]

\newtheorem{proposition}[theorem]{Proposition}

\numberwithin{equation}{section}
\makeatletter
	
	\@addtoreset{equation}{section}
\makeatother

\makeatletter
  
  \@addtoreset{figure}{section}
\makeatother

\makeatletter
  
  \@addtoreset{algorithm}{section}
\makeatother

\makeatletter
  
  \@addtoreset{table}{section}
\makeatother

\newfont{\bg}{cmr9 scaled\magstep4}
\newcommand{\bigzerol}{\smash{\lower1.0ex\hbox{\bg 0}}}

\DeclareMathOperator{\rank}{rank}
\DeclareMathOperator{\ind}{index}
\DeclareMathOperator{\diag}{diag}

\DeclareMathOperator*{\argmin}{argmin}

\begin{document}
\maketitle

\begin{abstract}
In this contribution, we study the numerical behavior of the Generalized Minimal Residual (GMRES) method for solving singular linear systems. 
It is known that GMRES determines a least squares solution without breakdown if the coefficient matrix is range-symmetric (EP), or if its range and nullspace are disjoint (GP) and the system is consistent.
We show that the accuracy of GMRES iterates may deteriorate in practice due to three distinct factors: (i)~the inconsistency of the linear system; (ii)~the distance of the initial residual to the nullspace of the coefficient matrix; (iii)~the extremal principal angles between the ranges of the coefficient matrix and its transpose.
These factors lead to poor conditioning of the extended Hessenberg matrix in the Arnoldi decomposition and affect the accuracy of the computed least squares solution.
We also compare GMRES with the range restricted GMRES (RR-GMRES) method. 
Numerical experiments show typical behaviors of GMRES for small problems with EP and GP matrices.
\end{abstract}

\section{Introduction} \label{sec:intro}
Consider solving linear systems of equations
\begin{align}
A \boldsymbol{x} = \boldsymbol{b},
\label{eq:Ax=b}
\end{align}
where $A \in \mathbb{R}^{n \times n}$ may be singular and $\boldsymbol{b} \in \mathbb{R}^n$ is not necessarily in $\mathcal{R}(A) = \lbrace \boldsymbol{y} \in \mathbb{R}^n \mid \boldsymbol{y} = A \boldsymbol{x},  \ \boldsymbol{x} \in \mathbb{R}^n\rbrace$, the range of $A$.
We say that $A \boldsymbol{x} = \boldsymbol{b}$ is consistent if $\boldsymbol{b} \in \mathcal{R}(A)$, and otherwise it is inconsistent.
If \cref{eq:Ax=b} is inconsistent, instead of \cref{eq:Ax=b}, it is natural to consider solving the least squares problem
\begin{align}
\| \boldsymbol{b} - A \boldsymbol{x} \| = \min_{\boldsymbol{u} \in \mathbb{R}^n} \| \boldsymbol{b} - A \boldsymbol{u} \|,
\label{eq:LPprob}
\end{align}
where $\| \cdot \|$ denotes the Euclidean norm.
We call a minimizer $\boldsymbol{x} \in \argmin_{\boldsymbol{u} \in \mathbb{R}^n} \| \boldsymbol{b} - A \boldsymbol{u} \|$ a least squares solution, which is not necessarily unique.

In order to analyze iterative methods for solving \cref{eq:Ax=b} in terms of the spaces associated with $A$, we give some required definitions and notations.
Let $\mathcal{N}(A) = \lbrace \boldsymbol{x} \in \mathbb{R}^n \mid A \boldsymbol{x} = \boldsymbol{0} \rbrace$ be the nullspace of $A$.
Then, we have $\mathcal{N}(A^\mathsf{T}) \oplus \mathcal{R}(A) = \mathcal{N}(A) \oplus \mathcal{R}(A^\mathsf{T}) = \mathbb{R}^n$, $\dim \mathcal{N}(A^\mathsf{T}) = \dim \mathcal{N}(A)$, and $\dim \mathcal{R}(A^\mathsf{T}) = \dim \mathcal{R}(A) = \rank (A)$, where $\oplus$ denotes the direct sum of subspaces.
Let $r = \rank (A)$ and denote the singular value decomposition (SVD) of $A$ by $U \Sigma V^\mathsf{T}$, where $U \in\mathbb{R}^{n \times n}$ and $V \in \mathbb{R}^{n \times n}$ are orthogonal matrices $U^\mathsf{T} \! U = U U^\mathsf{T} = V^\mathsf{T} V = V V^\mathsf{T} = \mathrm{I}$, $\mathrm{I}$ is the identity matrix, $\Sigma = \diag (\sigma_1, \sigma_2, \dots, \sigma_r, 0, 0, \dots, 0) \in \mathbb{R}^{n \times n}$, and $\sigma_i$ is the $i$th largest nonzero singular value of $A$.
Let $U = [U_1, U_2]$ and $V = [V_1, V_2]$, where the columns of $U_1 \in \mathbb{R}^{n \times r}$ and $U_2 \in \mathbb{R}^{n \times (n-r)}$ form orthonormal bases of $\mathcal{R}(A) = \mathcal{R}(U_1)$ and $\mathcal{R}(A)^\perp = \mathcal{N}(A^\mathsf{T}) = \mathcal{R}(U_2)$, respectively, and the columns of $V_1 \in \mathbb{R}^{n \times r}$ and $V_2 \in \mathbb{R}^{n \times (n-r)}$ form orthonormal bases of $\mathcal{N}(A)^\perp = \mathcal{R}(A^\mathsf{T}) = \mathcal{R}(V_1)$ and $\mathcal{N}(A) = \mathcal{R}(V_2)$, respectively.

We recall the definitions of generalized inverses.
We call a matrix $X \in \mathbb{R}^{n \times n}$ the Moore-Penrose generalized inverse (pseudoinverse) of $A \in \mathbb{R}^{n \times n}$ if $X$ satisfies the Penrose equations $A X A = A$, $X A X = X$, $(A X)^\mathsf{T} = A X$, and $(X A)^\mathsf{T} = XA$, denote it by $A^\dag$, and have the identity $A^\dag = V_1 \Sigma_r^{-1} {U_1}^\mathsf{T}$, where $\Sigma_r = \diag(\sigma_1, \sigma_2, \dots, \sigma_r)$.
The condition number of $A$ is denoted by $\kappa(A) = \| A \| \|A^\dag \|$ \cite[Definition 1.4.2]{Bjorck1996}.
The smallest nonnegative integer $k$ such that $\rank (A^k) = \rank (A^{k + 1})$ is called the index of $A$ \cite[Definition 7.2.1]{CampbellMeyer2009}, and is denoted by $\ind (A)$.
In addition, $k \geq \mathrm{index}(A) \iff \mathcal{N}(A^k) \oplus \mathcal{R}(A^k) = \mathbb{R}^n$ \cite[p.~121]{CampbellMeyer2009}.
Let $\ind(A) = 1$ and $X \in \mathbb{R}^{n \times n}$ be such that $A X A = A$, $X A X = X$, and $A X = X A$.
Then, $X$ is unique, and we call $X$ the group inverse of $A$ and denote it by $A^\#$.
The group inverse of the matrix $A$ is on $\mathcal{R}(A)$ equal to the inverse of the restriction of $A$ to its range ${\cal R}(A)$ and admits $\mathcal{N}(A)$ as its nullspace \cite[Theorem 2]{Robert1968}.
The group inverse can be characterized by the Jordan canonical form.
Let $S$ be a nonsingular matrix such that $J = S^{-1} \! A S$ is the Jordan canonical form of $A$.
Then, $A^\# = S J^\dag S^{-1}$ holds.
In particular, we have $\mathcal{R}(A^\#) = \mathcal{R}(A)$, $\mathcal{N}(A^\#) = \mathcal{N}(A)$, and $A^\# \! A = A A^\# = P_{\mathcal{R}(A), \mathcal{N}(A)}$ \cite{Ben-IsraelGreville2003}.
Here, $P_{\mathcal{R}(A), \mathcal{N}(A)}$ denotes the projection onto $\mathcal{R}(A)$ along $\mathcal{N}(A)$.

Now, we express solutions of \cref{eq:Ax=b,eq:LPprob}.
The vector $\boldsymbol{x}_* = A^\dag \boldsymbol{b} = V_1 \Sigma_r^{-1} U_1^\mathsf{T} \boldsymbol{b}$ is called the minimum-norm least squares or pseudoinverse solution of \cref{eq:Ax=b} or \cref{eq:LPprob}, and it belongs to $\mathcal{R}(A^\mathsf{T})$.
We next give the expressions of the residual of \cref{eq:LPprob}.
Denote the orthogonal projector onto $\mathcal{R}(A)$ by $P_{\mathcal{R}(A)} = U_1 U_1^\mathsf{T}$ and that onto $\mathcal{N}(A^\mathsf{T})$ by $P_{\mathcal{N}(A^\mathsf{T})} = U_2 U_2^\mathsf{T}$.
If $\boldsymbol{v} \rvert_\mathcal{S} \in \mathbb{R}^n$ is the orthogonal projection of a vector $\boldsymbol{v} \in \mathbb{R}^n$ onto the subspace $\mathcal{S} \subseteq \mathbb{R}^n$, then for any $\boldsymbol{x}_0 \in \mathbb{R}^n$ and any $\boldsymbol{b} \in \mathbb{R}^n$, the corresponding residual of \cref{eq:LPprob} is $\boldsymbol{r}_0 = \boldsymbol{b} - A \boldsymbol{x}_0 = \boldsymbol{b} \rvert_{\mathcal{N}(A^\mathsf{T})} + \boldsymbol{b} \rvert_{\mathcal{R}(A)} - A \boldsymbol{x}_0 = \boldsymbol{r}_* + \boldsymbol{r}_0 \rvert_{\mathcal{R}(A)}$, where $\boldsymbol{r}_* = \boldsymbol{b} \rvert_{\mathcal{N}(A^\mathsf{T})} = P_{\mathcal{N}(A^\mathsf{T})} \boldsymbol{b} \in \mathcal{N}(A^\mathsf{T})$ is the least squares residual and $\boldsymbol{b} \rvert_{\mathcal{R}(A)} = P_{\mathcal{R}(A)} \boldsymbol{b}$.
In particular, we have $\boldsymbol{r}_* = \boldsymbol{0}$ for $\boldsymbol{b} \in \mathcal{R}(A)$.
Assume $\ind (A)=1$.
It holds that $\boldsymbol{x}_\# = A^\# \boldsymbol{b}$ is a solution of $A \boldsymbol{x} = \boldsymbol{b}$ $\Longleftrightarrow$ $\boldsymbol{b} \in \mathcal{R}(A)$ \cite{CampbellMeyer2009}, and it is the unique solution of $A \boldsymbol{x} = \boldsymbol{b}$ in $\mathcal{R}(A)$.
Furthermore, the minimum Euclidean norm solution of \cref{eq:LPprob} satisfies $\boldsymbol{x}_* = P_{\mathcal{R(A^\mathsf{T})}} \boldsymbol{x}_\#$.

In this paper, we are interested in the numerical behavior of the Generalized Minimal Residual (GMRES) method \cite{SaadSchultz1986} applied in particular to singular systems \cref{eq:Ax=b}.
In \cref{sec:GMRsing}, we give some well-known conditions under which GMRES determines a solution without breakdown when applied to certain classes of singular matrices.
We discuss also a relation to the range-restricted GMRES (RR-GMRES) method proposed in \cite{CalvettiLewisReichel2000}.
In \cref{sec:GMR_EP}, we examine the conditioning of the coefficient matrix $A$ restricted to the Krylov subspaces that significantly influences the numerical behavior of GMRES.
We consider first the case of EP (equal projection) or range-symmetric matrices and distinguish between the consistent and inconsistent cases showing that the consistent case is similar to the nonsingular case.
Then we discuss the inconsistent EP case where GMRES suffers from an instability, since the convergence means ill-conditioned restriction of $A$ to the Krylov subspaces.
In \cref{sec:GMR_GP} we study the case of group (GP) matrices.
We show that the numerical behavior of GMRES applied to such problems depends substantially on the principal angles between the subspaces $\mathcal{R}(A)$ and $\mathcal{R}(A^\mathsf{T})$.
Surprisingly, difficulties can be expected for non-range-symmetric problems even for consistent systems.
In \cref{sec:conc}, we conclude the paper.

\section{GMRES methods and its convergence for singular systems} \label{sec:GMRsing}
GMRES for the linear system \cref{eq:Ax=b} with initial iterate $\boldsymbol{x}_0 \in \mathbb{R}^n$, independent of any particular implementation of the algorithm, determines the $k$th iterate $\boldsymbol{x}_k$ over $\boldsymbol{x}_0 + \mathcal{K}_k (A, \boldsymbol{r}_0)$ that minimizes $\| \boldsymbol{b} - A \boldsymbol{x}_k \|$, where $\boldsymbol{r}_0 = \boldsymbol{b} - A \boldsymbol{x}_0$ is the initial residual and $\mathcal{K}_k(A, \boldsymbol{r}_0) = \mathrm{span} \lbrace \boldsymbol{r}_0, A \boldsymbol{r}_0, \dots, A^{k-1} \boldsymbol{r}_0 \rbrace$ is the Krylov subspace of order $k$.
Note that there exist $\boldsymbol{x}_k$ and hence $\boldsymbol{r}_k = \boldsymbol{b} - A \boldsymbol{x}_k$ for all $k \geq 0$ but $\boldsymbol{x}_k$ may not be unique in the singular case.
Denote $\mathcal{K}_k(A, \boldsymbol{r}_0)$ by $\mathcal{K}_k$ for simplicity.
It is clear that $\mathcal{K}_k = \mathrm{span} \lbrace \boldsymbol{r}_0 \rbrace + A \mathcal{K}_{k-1} \subset \mathrm{span} \lbrace \boldsymbol{r}_* \rbrace + \mathcal{R}(A)$ holds.
If $\dim A \mathcal{K}_k = \dim \mathcal{K}_k$, then the problem
\begin{align}
\| \boldsymbol{b} - A \boldsymbol{x}_k \| = \min_{\boldsymbol{z} \in \mathcal{K}_k} \| \boldsymbol{b} - A (\boldsymbol{x}_0 + \boldsymbol{z}) \| = \min_{\boldsymbol{z} \in \mathcal{K}_k} \| \boldsymbol{r}_0 - A \boldsymbol{z} \| = \| \boldsymbol{r}_0 - A \boldsymbol{z}_k \|
\label{eq:prjLS}
\end{align}
has a unique solution $\boldsymbol{x}_k = \boldsymbol{x}_0 + \boldsymbol{z}_k$ and hence $\boldsymbol{r}_k = \boldsymbol{b} - A \boldsymbol{x}_k \in \boldsymbol{r}_0 + A \mathcal{K}_k$ is uniquely determined.

General studies on Krylov subspace methods in the singular case were done in \cite{IpsenMeyer1998}, \cite{Sidi1999}, \cite{WeiWu2000}, \cite{Schneider2005th}.
Particular studies on GMRES-type methods in the singular case were done in \cite{CalvettiLewisReichel2000}, \cite{Sidi2001}, \cite{ReichelYe2005}, \cite{Smoch2007}, \cite{DuSzyld2008}.
See \cite{CalvettiLewisReichel2002BIT}, \cite{CalvettiLewisReichel2002NumerMath}, \cite{EldenSimoncini2012} for GMRES on ill-posed linear systems, and \cite{Zhang2010} for GMRES with preconditioning.
See also \cite{GasparoPapiniPasquali2008} for GMRES and \cite{AwonoTagoudjeu2010} for GMRES with preconditioning in Hilbert spaces.

The GMRES method for solving ill-posed  problems was studied in \cite{CalvettiLewisReichel2002BIT}, \cite{CalvettiLewisReichel2002NumerMath}, \cite{EldenSimoncini2012}.
The preconditioned GMRES method applied to singular systems was considered in \cite{Zhang2010} .
See also \cite{GasparoPapiniPasquali2008} for GMRES and \cite{AwonoTagoudjeu2010} for the preconditioned GMRES method applied in Hilbert spaces.

In the nonsingular case, GMRES determines the solution of $A \boldsymbol{x} = \boldsymbol{b}$ for all $\boldsymbol{b} \in \mathbb{R}^n$ and for all $\boldsymbol{x}_0 \in \mathbb{R}^n$ within $n$ iterations.
In the singular case, GMRES may fail to determine a solution of \cref{eq:Ax=b}, and is said to break down at some step $k$ if $\dim A \mathcal{K}_k < \dim \mathcal{K}_k ~ \mbox{or} ~ \dim \mathcal{K}_k < k$ \cite[p.~38]{BrownWalker1997}.
Note that, in general, $\dim A \mathcal{K}_k \leq \dim \mathcal{K}_k \leq k$ holds for each $k$.

We give an explicit expression of the iterate $\boldsymbol{x}_k$ for GMRES using the Arnoldi decomposition $A Q_k = Q_{k+1} H_{k+1, k}$, $k = 1, 2, \dots$, where the columns of $Q_k = \left[ \boldsymbol{q}_1, \boldsymbol{q}_2, \dots, \boldsymbol{q}_k \right]$ form an orthonormal basis of the Krylov subspace $\mathcal{K}_k$, and $H_{k+1, k} = (h_{i,j}) \in \mathbb{R}^{(k+1) \times k}$ is an extended Hessenberg matrix.
Then the iterate is given by $\boldsymbol{x}_k = \boldsymbol{x}_0 + Q_k \boldsymbol{y}_k$ with $\boldsymbol{y}_k = \mathrm{arg\,min}_{\boldsymbol{y} \in \mathbb{R}^k} \| \beta \boldsymbol{e}_1 - H_{k+1, k} \boldsymbol{y} \|$, where $\beta = \| \boldsymbol{r}_0 \|$, $\boldsymbol{e}_1$ is the first column of the identity matrix and $\| \boldsymbol{b} - A \boldsymbol{x}_k \| = \| \boldsymbol{r}_0 - A Q_k \boldsymbol{y}_k \| = \| \beta \boldsymbol{e}_1 - H_{k+1,k} \boldsymbol{y}_k \|$.

It is clear that if $h_{i+1,i} \ne 0$ for $i=1, 2, \dots, k-1$, the breakdown does not occur until step $k-1$ of GMRES with $\dim A \mathcal{K}_i = i$, or $\rank(H_{i+1, i}) = i$, $i = 1, 2, \dots, k-1$.
At breakdown of GMRES at step $k$ with $h_{k+1,k}=0$, one of the following cases holds \cite[Appendix B]{MorikuniHayami2015} (cf.\ \cite[Theorem~2.2]{BrownWalker1997}):
\begin{description}[style=unboxed,leftmargin=0cm]
	\item[Case I.] $\dim A \mathcal{K}_{k+1} = k < \dim \mathcal{K}_{k+1} = k + 1$, whereas $\rank (H_{k,k}) = k-1$.
	\item[Case II.] $\dim A \mathcal{K}_k = k = \dim \mathcal{K}_{k+1} < k+1$, whereas $\rank(H_{k, k}) = k$ (GMRES determines a solution of $A \boldsymbol{x} = \boldsymbol{b}$ at step $k$).
\end{description}
Here, $H_{k.k} = (h_{i,j}) \in \mathbf{R}^{k \times k}$.

A variant of GMRES called the range restricted GMRES (RR-GMRES) method was proposed in \cite{CalvettiLewisReichel2000}.
RR-GMRES determines the $k$th iterate by minimizing the same objective function as GMRES over a different Krylov subspace
\begin{align*}
\| \boldsymbol{b} - A \boldsymbol{x}_k^\mathrm{R} \| & = \min_{\boldsymbol{z} \in \mathcal{K}_k (A, A \boldsymbol{r}_0)} \| \boldsymbol{b} - A (\boldsymbol{x}_0 + \boldsymbol{z}) \| = \min_{\boldsymbol{z} \in \mathcal{K}_k (A, A \boldsymbol{r}_0)} \| \boldsymbol{r}_0 - A \boldsymbol{z} \| = \| \boldsymbol{r}_0 - A \boldsymbol{z}_k^\mathrm{R} \|.
\end{align*}
It was shown in \cite[Theorem~A2]{CaoWang2002} that if RR-GMRES applied to \cref{eq:LPprob} breaks down at step $m$ with $\rank (A) = m-1$ and $\dim A \mathcal{K}_m (A, A\boldsymbol{r}_0) = m-1$, then it determines a solution of \cref{eq:LPprob}.
Here, RR-GMRES is said to break down if $\dim A \mathcal{K}_k (A, A \boldsymbol{r}_0) < \dim \mathcal{K}_k (A, A \boldsymbol{r}_0)$ or $\dim \mathcal{K}_k (A, A \boldsymbol{r}_0) < k$.

We give an explicit expression of the RR-GMRES iterate $\boldsymbol{x}_k$ using the Arnoldi decomposition $A Q_k^\mathrm{R} = Q_{k+1}^\mathrm{R} H_{k+1, k}^\mathrm{R}$, $k = 1, 2, \dots$, where the columns of $Q_k^\mathrm{R} = [ \boldsymbol{q}_1^\mathrm{R}, \boldsymbol{q}_2^\mathrm{R}, \dots, \boldsymbol{q}_k^\mathrm{R} ]$ form an orthonormal basis of the Krylov subspace $\mathcal{K}_k (A, A \boldsymbol{r}_0)$ with the initial vector $\boldsymbol{q}_1^\mathrm{R} = A \boldsymbol{r}_0 / \| A \boldsymbol{r}_0 \|$, and $H_{k+1, k}^\mathrm{R} = ( h_{i, j}^\mathrm{R} ) \in \mathbb{R}^{(k+1) \times k}$ is an extended Hessenberg matrix.
Then, the iterate is given by $\boldsymbol{x}_k^\mathrm{R} = \boldsymbol{x}_0 + Q_k^\mathrm{R} \boldsymbol{y}_k^\mathrm{R}$ with $\boldsymbol{y}_k^\mathrm{R} = \mathrm{arg\,min}_{\boldsymbol{y} \in \mathbb{R}^k} \| (Q_{k+1}^\mathrm{R})^\mathsf{T} \boldsymbol{r}_0 - H_{k+1, k}^\mathrm{R} \boldsymbol{y} \|$, where
\begin{align*}
\| \boldsymbol{b} - A \boldsymbol{x}_k^\mathrm{R} \|^2 & = \| \boldsymbol{r}_0 - A Q_k^\mathrm{R} \boldsymbol{y}_k^\mathrm{R} \|^2 \\
& = \| (Q_{k+1}^\mathrm{R})^\mathsf{T} \boldsymbol{r}_0 - H_{k+1,k}^\mathrm{R} \boldsymbol{y}_k^\mathrm{R} \|^2 + \| [\mathrm{I} - Q_{k+1}^\mathrm{R} (Q_{k+1}^\mathrm{R})^\mathsf{T}] \boldsymbol{r}_0 \|^2 \\
& = \min_{\boldsymbol{y} \in \mathbb{R}^k} \| (Q_{k+1}^\mathrm{R})^\mathsf{T} \boldsymbol{r}_0 - H_{k+1,k}^\mathrm{R} \boldsymbol{y} \|^2 + \| [\mathrm{I} - Q_{k+1}^\mathrm{R} (Q_{k+1}^\mathrm{R})^\mathsf{T}] \boldsymbol{r}_0 \|^2.
\end{align*}
The last term is equal to the $k$th residual norm for the simpler GMRES method \cite{WalkerZhou1994}, which is not larger than the $k$th residual norm for RR-GMRES, i.e., $\| \boldsymbol{b} - A \boldsymbol{x}_k \| \leq \| \boldsymbol{b} - A \boldsymbol{x}_k^\mathrm{R} \|$. Note also that $\| H_{k,k-1}^\mathrm{R} \| = \| A Q_{k-1}^\mathrm{R} \| \leq \| A Q_k \| = \| H_{k+1,k} \|$
and
\begin{align*}
\sigma_k (H_{k+1, k}) \leq \min_{\boldsymbol{y} \in \mathbb{R}^{k-1} \backslash \lbrace \boldsymbol{0} \rbrace} \frac{\| A Q_{k-1}^\mathrm{R} \boldsymbol{y} \|}{\| Q_{k-1}^\mathrm{R} \boldsymbol{y} \|} = \sigma_{k-1} (H_{k, k-1}^\mathrm{R})
\end{align*}
leading to an interesting bound $\kappa(H_{k, k-1}^\mathrm{R}) \leq \kappa(H_{k+1, k})$ for $k=2, \dots, n-1$, where $\sigma_k (\cdot)$ denotes the $k$th largest singular value of a matrix.

\bigskip

In the following, we present conditions under which GMRES determines a solution of \cref{eq:Ax=b}.
We start with the observation that in the case of $\mathcal{N}(A) \cap \mathcal{R}(A) \ne\lbrace \boldsymbol{0} \rbrace$, GMRES breaks down and fails to determine a solution.

\begin{proposition} \label{prop:bd}
	If $\boldsymbol{b} \in \mathcal{R}(A)$ and $\boldsymbol{0} \ne\boldsymbol{r}_0 \in \mathcal{N}(A) \cap \mathcal{R}(A)$, then GMRES breaks down at step 1 without determining a solution of $A \boldsymbol{x} = \boldsymbol{b}$.
\end{proposition}
\begin{proof}
	Since $\boldsymbol{r}_0 \ne\boldsymbol{0}$, we have $\dim \mathcal{K}_1 = \dim \mathrm{span} \lbrace \boldsymbol{r}_0 \rbrace = 1$.
	Since $\boldsymbol{r}_0 \in \mathcal{N}(A)$ gives $A \boldsymbol{r}_0 = \boldsymbol{0}$, we have $\dim A \mathcal{K}_1 = \dim \mathrm{span} \lbrace A \boldsymbol{r}_0 \rbrace = \dim \mathrm{span} \lbrace \boldsymbol{0} \rbrace = 0$.
	Hence, $\dim A \mathcal{K}_1 < \dim \mathcal{K}_1$ holds.
	Therefore, GMRES breaks down at step 1 without determining a solution of $A \boldsymbol{x} = \boldsymbol{b}$.
\end{proof}
Similarly to GMRES, RR-GMRES also breaks down at step 1 without determining a solution of $A \boldsymbol{x} = \boldsymbol{b}$ if $\boldsymbol{b} \in \mathcal{R}(A)$ and $\boldsymbol{0} \ne\boldsymbol{r}_0 \in \mathcal{N}(A) \cap \mathcal{R}(A)$.
Therefore, we will restrict our attention to the cases of $\mathcal{N}(A) \cap \mathcal{R}(A) = \lbrace \boldsymbol{0} \rbrace$.
The following statement holds.

\begin{theorem}[{\cite[Theorem~2.6]{BrownWalker1997}, \cite[Theorem~3.2]{WeiWu2000}}] \label{th:GMRESconv_SS}
	If $\mathcal{N}(A) \cap \mathcal{R}(A) = \lbrace \boldsymbol{0} \rbrace$, then GMRES determines a solution of $A \boldsymbol{x} = \boldsymbol{b}$ without breakdown for all $\boldsymbol{b} \in \mathcal{R}(A)$ and for all $\boldsymbol{x}_0 \in \mathbb{R}^n$.
	The solution is $\boldsymbol{x}_\#  + (\mathrm{I} - A^\# \! A) \boldsymbol{x}_0$.
\end{theorem}

The condition $\mathcal{N}(A) \cap \mathcal{R}(A) = \lbrace \boldsymbol{0} \rbrace$ is equivalent to $\rank ([U_1, V_2]) = n$, or $\mathcal{R}(U_1) \cap \mathcal{R}(V_2) = \lbrace \boldsymbol{0} \rbrace$ \cite[Lemma~7.2.1]{CampbellMeyer2009}.
Then, we have $P_{\mathcal{R}(A), \mathcal{N}(A)} = A^\# \! A = ff[U_1, \mathrm{O}] [U_1, V_2]^{-1} = \break U_1 (V_1^\mathsf{T} U_1)^{-1} V_1^\mathsf{T}$, because of \cite[Exercise 30, p.~167]{Ben-IsraelGreville2003} and
\begin{align*}
[U_1, V_2]^{-1} =
\begin{bmatrix}
(V_1^\mathsf{T} U_1)^{-1} & \mathrm{O} \\
- V_2 U_1 (V_1^\mathsf{T} U_1)^{-1} & \mathrm{I}
\end{bmatrix}
[V_1, V_2]^\mathsf{T}.
\end{align*}
Thus, $\kappa (A^\# \! A) = \kappa (V_1^\mathsf{T} U_1)$ holds.

\bigskip

For a special class of singular matrices, GMRES determines a least squares solution.

\begin{theorem}[{\cite[Theorem~2.4]{BrownWalker1997}}] \label{th:GMRESconv_BW}
	If $\mathcal{R}(A) = \mathcal{R}(A^\mathsf{T})$, then GMRES determines a solution of $\min_{\boldsymbol{x} \in \mathbb{R}^n} \| \boldsymbol{b}  - A \boldsymbol{x} \|$ without breakdown for all $\boldsymbol{b} \in \mathbb{R}^n$ and for all $\boldsymbol{x}_0 \in \mathbb{R}^n$.
\end{theorem}

A matrix $A \in \mathbb{R}^{n \times n}$ satisfying $\mathcal{N}(A) \cap \mathcal{R}(A) = \lbrace \boldsymbol{0} \rbrace$, is called a GP (group) matrix. A  GP matrix satisfying in addition $\mathcal{R}(A^\mathsf{T}) = \mathcal{R}(A)$, or equivalently $\mathcal{R}(U_1) = \mathcal{R}(V_1)$, is called an EP (equal projection) or range-symmetric matrix.
Now, we characterize the GP and EP matrices in terms of their singular value decompositions.
The matrix $A$ can be decomposed into
\begin{align*}
A =
U
\begin{bmatrix}
\Sigma_r K & \Sigma_r L \\
\mathrm{O} & \mathrm{O}
\end{bmatrix}
U^\mathsf{T}
=
V
\begin{bmatrix}
K \Sigma_r & \mathrm{O} \\
M \Sigma_r & \mathrm{O}
\end{bmatrix}
V^\mathsf{T}
\end{align*}
with the identity $K K^\mathsf{T} + L L^\mathsf{T} = \mathrm{I}$, where $K = V_1^\mathsf{T} \! U_2$, $L = V_1^\mathsf{T} U_1$, and $M = V_2^\mathsf{T} U_1$, and the pseudoinverse matrix of $A$ can be decomposed into
\begin{align*}
A^\dag = U 
\begin{bmatrix}
K^\mathsf{T} \Sigma_r^{-1} & \mathrm{O} \\
L^\mathsf{T} \Sigma_r^{-1} & \mathrm{O}
\end{bmatrix}
U^\mathsf{T}
= V
\begin{bmatrix}
\Sigma_r^{-1} \! K^\mathsf{T} & \Sigma_r^{-1} \! M^\mathsf{T} \\
\mathrm{O} & \mathrm{O}
\end{bmatrix}
V^\mathsf{T}.
\end{align*}
The equivalences for GP matrices
\begin{align*}
\mathcal{R}(U_1)  \cap \mathcal{R}(V_2) = \lbrace \boldsymbol{0} \rbrace \Longleftrightarrow \rank ([U_1, V_2]) = n \Longleftrightarrow V_1^\mathsf{T} \! U_1 \mbox{($= K$) is nonsingular}
\end{align*}
follow from the equation
\begin{align*}
[V_1, V_2]^\mathsf{T} [U_1, V_2] =
\begin{bmatrix}
V_1^\mathsf{T} U_1 & \mathrm{O} \\
V_2^\mathsf{T} U_1 & \mathrm{I}
\end{bmatrix}.
\end{align*}
The group inverse of a GP matrix $A$ can be decomposed into
\begin{align*}
A^\# = U
\begin{bmatrix}
K^{-1} \Sigma_r^{-1} & K^{-1} \Sigma_r^{-1} K^{-1} L \\
\mathrm{O} & \mathrm{O}
\end{bmatrix}
U^\mathsf{T} 
= V 
\begin{bmatrix}
\Sigma_r^{-1} \! K^{-1} & \mathrm{O} \\
M K^{-1} \Sigma_r^{-1} \! K^{-1} & \mathrm{O}
\end{bmatrix}
V^\mathsf{T}.
\end{align*}
See \cite[Section 1]{BaksalaryTrenkler2008}.
The conditioning of $K = V_1^\mathsf{T} \! U_1$ is independent of the conditioning of $A$ but it gives a difficulty in solving singular linear systems with GMRES.
The EP case $\mathcal{R}(A^\mathsf{T}) = \mathcal{R}(A)$ is equivalent to that $K = V_1^\mathsf{T} U_1$ is orthogonal, since $L = V_1^\mathsf{T} U_2 = \mathrm{O}$.

Next, we characterize GP and EP matrices in terms of the principal angles.
In the EP case $\mathcal{R}(A^\mathsf{T}) = \mathcal{R}(A)$, the matrix $V_1^\mathsf{T} U_1$ is orthogonal and the cosines of the principal angles between $\mathcal{R}(A)$ and $\mathcal{R}(A^\mathsf{T})$ are all zero.
In the GP case, since the columns of $U_1$ and $V_1$ form bases of $\mathcal{R}(A)$ and $\mathcal{R}(A^\mathsf{T})$, respectively, the cosines of the canonical angles between $\mathcal{R}(A)$ and $\mathcal{R}(A^\mathsf{T})$ are the singular values of $V_1^\mathsf{T} U_1$ \cite[section 1.2]{Chatelin2012}.
Hence, the condition number of $V_1^\mathsf{T} \! U_1$ is related to the extremal principal angles.

Note that due to $\| V_1^\mathsf{T} U_1 \| \leq 1$ all singular values of $V_1^\mathsf{T} U_1$ are less than or equal to 1 and the number of those equal exactly to 1 gives the dimension of
$\mathcal{R}(U_1) \cap \mathcal{R}(V_1)$. So, if  $\| V_1^\mathsf{T} U_1 \| < 1$, then  $\mathcal{R}(U_1) \cap \mathcal{R}(V_1) =
\mathcal{R}(A) \cap \mathcal{R}(A^\mathsf{T}) = \lbrace \boldsymbol{0} \rbrace$. If a matrix $A \in \mathbb{R}^{n \times n}$ 
satisfies $\mathcal{R}(A) \cap \mathcal{R}(A^\mathsf{T}) = \lbrace \boldsymbol{0} \rbrace$, it is called a disjoint range (DR) matrix \cite{BaksalaryTrenkler2011}.

\section{GMRES and EP matrices} \label{sec:GMR_EP}
As was already noted, the GMRES iterate $\boldsymbol{x}_k = \boldsymbol{x}_0 + \boldsymbol{z}_k$ solves the least squares problem \cref{eq:prjLS}.
Therefore, the restriction of $A$ to the Krylov subspace $\mathcal{K}_k \subseteq \mathbb{R}^n$ denoted by $A \rvert_{\mathcal{K}_k}$ plays an important role in the numerical behavior of GMRES.
Indeed, the ill-conditioning of $A \rvert_{\mathcal{K}_k}$ was studied and its condition number
\begin{align*}
\kappa(A \rvert_{\mathcal{K}_k}) = \frac{\max_{\boldsymbol{z} \in \mathcal{K}_k \backslash \lbrace \boldsymbol{0} \rbrace} \| A \boldsymbol{z} \| \mbox{\Big /} \| \boldsymbol{z} \|}{\min_{\boldsymbol{z} \in \mathcal{K}_k \backslash \lbrace \boldsymbol{0} \rbrace} \| A \boldsymbol{z} \| \mbox{\Big /} \| \boldsymbol{z} \|}
\end{align*}
was introduced by Brown and Walker in \cite{BrownWalker1997}.
In practical computations, the iterate $\boldsymbol{x}_k$ is computed as $\boldsymbol{x}_k = \boldsymbol{x}_0 + Q_k \boldsymbol{y}_k$, where the columns of $Q_k$ form an orthonormal basis of the Krylov subspace $\mathcal{K}_k$ and the vector $\boldsymbol{y}_k$ is a solution of the extended Hessenberg least squares problem $\min_{\boldsymbol{y} \in \mathbb{R}^k} \| \beta \boldsymbol{e}_1 - H_{k+1, k} \boldsymbol{y} \|$ (see \cref{sec:GMRsing}).
The accuracy of $\boldsymbol{x}_k$ is thus affected directly by the conditioning of the matrix $H_{k+1, k}$, whereas the identity $\kappa(H_{k+1, k}) = \kappa(A \rvert_{\mathcal{K}_k})$ follows from the identities 
\begin{align*}
\lbrace \mathop{\max, \min }\limits_{\boldsymbol{z} \in \mathcal{K}_k \backslash \lbrace \boldsymbol{0} \rbrace} \rbrace \frac{\| A \boldsymbol{z} \|}{\| \boldsymbol{z} \|} = \lbrace \mathop{\max, \min }\limits_{\boldsymbol{w} \in \mathbb{R}^k \backslash \lbrace \boldsymbol{0} \rbrace} \rbrace\frac{\| A Q_k \boldsymbol{w} \|}{\| Q_k \boldsymbol{w} \|} = \lbrace \mathop{\max, \min }\limits_{\boldsymbol{w} \in \mathbb{R}^k \backslash \lbrace \boldsymbol{0} \rbrace} \rbrace \frac{\| H_{k+1,k} \boldsymbol{w} \|}{\| \boldsymbol{w} \|}.
\end{align*}

Next, we give bounds on the extremal singular values of $H_{k+1, k}$.
The norm of the matrix $H_{k+1,k}$ can be always bounded above by that of $A$
\begin{align*}
\| H_{k+1,k} \| = \max_{\boldsymbol{z} \in \mathcal{K}_k \backslash \lbrace \boldsymbol{0} \rbrace} \frac{\| A \boldsymbol{z} \|}{\| \boldsymbol{z} \|} \leq \max_{\boldsymbol{z} \in \mathrm{span} \lbrace \boldsymbol{r}_* \rbrace \oplus \mathcal{R}(A) \backslash \lbrace \boldsymbol{0} \rbrace} \frac{\| A \boldsymbol{z} \|}{\| \boldsymbol{z} \|} \leq \max_{\boldsymbol{z} \in \mathbb{R}^n \backslash \lbrace \boldsymbol{0} \rbrace} \frac{\| A \boldsymbol{z} \|}{\| \boldsymbol{z} \|} = \| A \|.
\end{align*}
This approach cannot be used to bound the $k$th (or smallest) singular value of $H_{k+1, k}$ due to 
\begin{align}
\sigma_k (H_{k+1,k}) = \min_{\boldsymbol{z} \in \mathcal{K}_k \backslash \lbrace \boldsymbol{0} \rbrace}  \frac{\| A \boldsymbol{z} \|}{\| \boldsymbol{z} \|} \geq \min_{\boldsymbol{z} \in \mathrm{span} \lbrace \boldsymbol{r}_* \rbrace \oplus \mathcal{R}(A) \backslash \lbrace \boldsymbol{0} \rbrace } \frac{\| A \boldsymbol{z} \|}{\| \boldsymbol{z} \|} \geq \min_{z \in \mathbb{R}^n \backslash \lbrace \boldsymbol{0} \rbrace} \frac{\| A \boldsymbol{z} \|}{\| \boldsymbol{z} \|} = 0
\label{eq:lbH}
\end{align}
as the last equality holds for $A$ singular.

In the consistent case, the condition number is bounded by $\kappa(H_{k+1,k}) \leq \kappa(A \rvert_{\mathcal{R}(A)})$ from $\mathcal{K}_k \subseteq \mathcal{R}(A)$ and
\begin{align*}
\sigma_k(H_{k+1,k}) \geq \min_{\boldsymbol{z} \in \mathcal{R}(A) \backslash \lbrace \boldsymbol{0} \rbrace} \frac{\| A \boldsymbol{z} \|}{\| \boldsymbol{z} \|},
\end{align*}
where $A \rvert_{\mathcal{R}(A)}$ denotes the restriction of $A$ to the range $\mathcal{R} (A)$.
If $A$ is an EP matrix $\mathcal{R}(A) = \mathcal{R}(A^\mathsf{T}) = \mathcal{N}(A)^{\perp}$, then 
\begin{align*}
\min_{\boldsymbol{z} \in \mathcal{R}(A) \backslash \lbrace \boldsymbol{0} \rbrace} \frac{\| A \boldsymbol{z} \|}{\| \boldsymbol{z} \|} = \min_{\boldsymbol{z} \in \mathcal{R}(A^\mathsf{T}) \backslash \lbrace \boldsymbol{0} \rbrace} \frac{\| A \boldsymbol{z} \|}{\| \boldsymbol{z} \|} = \sigma_r (A) > 0
\end{align*}
and
\begin{align*}
\kappa(A \rvert_{\mathcal{R}(A)}) = \frac{\| A \|}{\min_{\boldsymbol{z} \in \mathcal{R}(A) \backslash \lbrace \boldsymbol{0} \rbrace} \| A \boldsymbol{z} \| / \| \boldsymbol{z} \|}= \kappa(A).
\end{align*}
Indeed, the consistent EP case is similar to the nonsingular case, and the condition number of the extended Hessenberg matrix $H_{k+1, k}$ is bounded by $\kappa(H_{k+1,k}) \leq \kappa(A)$ (cf.\ \cite[Remark 3.2, Theorem 3.6]{ZhangWei2004}).
Consequently, the rank deficiency of the least squares problem \cref{eq:prjLS} cannot occur and GMRES will terminate if a solution is reached at some step with a degeneracy of the Krylov subspace at the next step.

In the inconsistent EP case, the equivalence  $\mathcal{R}(A^\mathsf{T}) = \mathcal{R}(A)$ $\Longleftrightarrow$ $\mathcal{N}(A^\mathsf{T}) = \mathcal{N}(A)$ shows that the nonzero least squares residual $\boldsymbol{r}_* \in \mathcal{N}(A^\mathsf{T})$ belongs also to $\mathcal{N}(A)$ and 
\begin{align*}
\sigma_k (H_{k+1,k}) \geq \min_{\boldsymbol{z} \in \mathrm{span} \lbrace \boldsymbol{r}_* \rbrace \oplus \mathcal{R}(A) \backslash \lbrace \boldsymbol{0} \rbrace } \frac{\| A \boldsymbol{z} \|}{\| \boldsymbol{z} \|} = 0.
\end{align*}
It follows from \cref{eq:prjLS} that  the residual $\boldsymbol{r}_{k-1}$ at step $k-1$ belongs to the Krylov subspace $\mathcal{K}_k$ and satisfies $\boldsymbol{r}_{k-1} - \boldsymbol{r}_* \in \mathcal{R}(A)$.
In addition, due to $A \boldsymbol{r}_* = \boldsymbol{0}$ we have
\begin{align}
\sigma_k (H_{k+1, k}) & = \min_{\boldsymbol{z} \in \mathcal{K}_k \backslash \lbrace \boldsymbol{0} \rbrace} \frac{\| A \boldsymbol{z} \|}{\| \boldsymbol{z} \|} \leq \frac{\| A \boldsymbol{r}_{k-1} \|}{\| \boldsymbol{r}_{k-1} \|} \notag \\
& = \frac{\| A (\boldsymbol{r}_{k-1} - \boldsymbol{r}_*) \|}{\| \boldsymbol{r}_{k-1} \|} \leq \| A \| \frac{\| \boldsymbol{r}_{k-1} - \boldsymbol{r}_* \|}{\| \boldsymbol{r}_{k-1} \|}.
\label{eq:ubArr}
\end{align}
This result was derived in a somewhat different form in \cite[Theorem~2.5]{BrownWalker1997}.
It is clear that in the inconsistent case, the least squares problem \cref{eq:prjLS} becomes ill-conditioned as the GMRES iterate converges to a least squares solution.
This situation is illustrated in \cref{fig:r_EP}.
Note that we also have
\begin{align*}
\sigma_k (H_{k+1, k}) = \min_{\boldsymbol{z} \in \mathcal{K}_k \backslash \lbrace \boldsymbol{0} \rbrace} \frac{\| A \boldsymbol{z} \|}{\| \boldsymbol{z} \|} \leq \frac{\| A \boldsymbol{r}_0 \|}{\| \boldsymbol{r}_0 \|} \leq \frac{\| A \| \| \boldsymbol{r}_0 \rvert_{\mathcal{R}(A^\mathsf{T})} \|}{\| \boldsymbol{r}_0 \|}
=  \frac{\| A \| \| \boldsymbol{r}_0 \rvert_{\mathcal{R}(A)} \|}{\| \boldsymbol{r}_0 \|}.
\end{align*}
This bound indicates that if the norm of  $A \boldsymbol{r}_0$ is too small, then $H_{k+1, k}$ becomes ill-conditioned and the inaccuracy can be expected at all subsequent steps of GMRES.
Finally, since a symmetric matrix is an EP matrix, the above discussion also covers the MINRES method \cite{PaigeSaunders1975} applied to symmetric singular systems.

The conditioning of the extended Hessenberg matrix $H_{k+1, k}$ for GMRES and its relation to the conditioning of $A$ are illustrated on small examples.
First for simplicity, consider applying GMRES with $\boldsymbol{x}_0 = \boldsymbol{0}$ to $A \boldsymbol{x} = \boldsymbol{b}$, where
\begin{align}
A =
\begin{bmatrix}
1 & 0 \\
0 & 0
\end{bmatrix}, \quad
\boldsymbol{b} =
\begin{bmatrix}
1 \\
\varepsilon
\end{bmatrix}, \quad
\varepsilon > 0.
\label{eq:2x2exactEP}
\end{align}
The matrix $A$ is EP, it has the range $\mathcal{R}(A) = \mathcal{R}(A^\mathsf{T}) = \mathrm{span} \lbrace [1, 0]^\mathsf{T} \rbrace$ and the nullspace $\mathcal{N}(A) = \mathcal{N}(A^\mathsf{T}) = \mathrm{span} \lbrace [0, 1]^\mathsf{T} \rbrace$, and its minimum nonzero singular value is $\sigma_1(A) = \break \min_{\boldsymbol{z} \in \mathcal{R}(A) \backslash \lbrace \boldsymbol{0} \rbrace} \| A \boldsymbol{z} \| / \| \boldsymbol{z} \| = 1$.
The first two steps of the Arnoldi process for $A$ and the initial vector $\boldsymbol{q}_1 = \boldsymbol{b} / \| \boldsymbol{b} \|$ give the decomposition $A Q_2 = Q_2 H_{2, 2} $, where
\begin{align*}
Q_2= 	[\boldsymbol{q}_1, \boldsymbol{q}_2 ]  = \frac{1}{\sqrt{1+\varepsilon^2}}
\begin{bmatrix}
1 &   \varepsilon \\
\varepsilon & -1
\end{bmatrix}, \quad
H_{2,2} = \frac{1}{1+\varepsilon^2}
\begin{bmatrix}
1 & \varepsilon \\
\varepsilon & \varepsilon^2
\end{bmatrix}.
\end{align*}
Hence, $H_{2, 2}$ is singular and we have $\sigma_1 (A) = \sigma_1 (H_{2, 2}) = 1$.
Solving \break $\min_{\boldsymbol{y} \in \mathbb{R}^2} \| \beta \boldsymbol{e}_1 - H_{2, 2} \boldsymbol{y} \|$, where $\beta = \sqrt{1+ \varepsilon^2}$, we have the minimum-norm solution $\boldsymbol{y}_2 = \break 1/\sqrt{1+\varepsilon^2} [1, \varepsilon]^\mathsf{T}$ and $\| \boldsymbol{y}_2 \| = 1$.
Therefore, $\| \boldsymbol{x}_2 \| = \| \boldsymbol{y}_2 \|$ and thus the norm of the iterate is not large, even if $\varepsilon$ is very small.
It is also clear that for $\varepsilon=0$ the system \cref{eq:2x2exactEP} becomes consistent and then GMRES will deliver the minimum norm solution $\boldsymbol{x}_* = [1, 0]^\mathsf{T}$ in one iteration.

For comparison, we also consider the $2 \times 2$ nonsingular ill-posed linear system $A \boldsymbol{x} = \boldsymbol{b}$, where 
\begin{align}
A = 
\begin{bmatrix}
1 & 0 \\
0 & \delta
\end{bmatrix},
\quad
\boldsymbol{b} = 
\begin{bmatrix}
1\\
\varepsilon
\end{bmatrix},
\quad
\delta, \varepsilon > 0
\label{eq:2x2nearEP}
\end{align}
with the condition number $\kappa(A)= 1 / \delta$, where $\delta$ is a small scalar and the right-hand side $\boldsymbol{b}$ is contaminated by the error $[0, \varepsilon]^\mathsf{T}$.
The exact solution is $[1, 0]^\mathsf{T}$ and the error-contaminated solution is $[1, \varepsilon/\delta]^\mathsf{T}$.
GMRES applied to $A \boldsymbol{x} = \boldsymbol{b}$ with $\boldsymbol{x}_0 = \boldsymbol{0}$ gives the same $Q_2$ as the one for \cref{eq:2x2exactEP} but the different Hessenberg matrix
\begin{align*}
H_{2, 2} = \frac{1}{1 + \varepsilon^2} 
\begin{bmatrix}
1 + \delta \varepsilon^2 & (1-\delta) \varepsilon\\
(1-\delta ) \varepsilon & \delta + \varepsilon^2
\end{bmatrix},
\end{align*}
whose condition number is the same as that of $A$.
\begin{figure}[t]
	\centering
	\includegraphics[width=160pt]{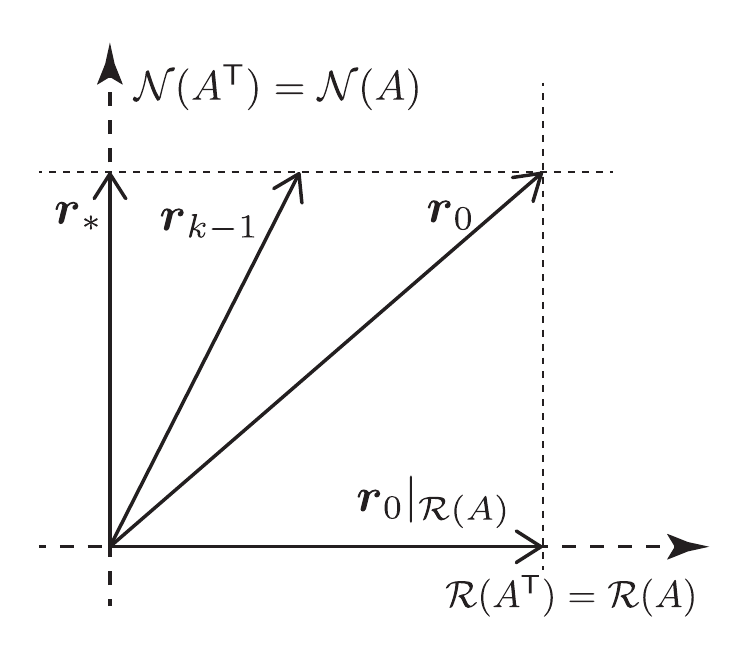}
	\caption{Geometric illustration of residual vectors in the EP case.}
	\label{fig:r_EP}
\end{figure}
Thus, we have the upper bound $\| A^{-1} \boldsymbol{b} \| = \| \boldsymbol{x}_2 \|=\| \boldsymbol{y}_2 \| \leq \| A^{-1} \| \| \boldsymbol{b} \| = \| H_{2, 2}^{-1} \| \| \boldsymbol{r}_0 \| = \sqrt{1+\varepsilon^2} / \delta$.
This means that if $\varepsilon \ne 0$ and $\delta$ is small, then the iterates computed by GMRES can be inaccurate.
In the exactly singular case with $\delta = 0$, this is not the case. 

In the following numerical examples, we examine the accuracy of the GMRES iterate with respect to the degree of consistency of linear systems by using the test matrix and right hand side vectors
\begin{align}
A = \begin{bmatrix} D & 0 \\ 0 & 0 \end{bmatrix} \in \mathbb{R}^{128 \times 128}, \quad \boldsymbol{b} =
\begin{bmatrix}
\boldsymbol{\gamma} \\
\boldsymbol{\delta}
\end{bmatrix},
\label{eq:test_inc}
\end{align}
where $D = \diag (10^{\frac{0}{63}}, 10^{\frac{- 4}{63}}, 10^{\frac{- 8}{63}}, \dots, 10^{-4}) \in \mathbb{R}^{64 \times 64}$,  $\boldsymbol{\gamma} = [\gamma, \gamma, \dots, \gamma]^\mathsf{T} \in \mathbb{R}^{64}$ and $\boldsymbol{\delta} = \left[\delta, \delta, \dots, \delta\right]^\mathsf{T} \in \mathbb{R}^{64}$.
Hence, $A$ has the condition number $10^4$, and $\boldsymbol{b} \not \in \mathcal{R}(A)$ $\Longleftrightarrow$ $\delta \ne0$.
The degree of inconsistency of the linear system $A \boldsymbol{x} = \boldsymbol{b}$ can be controlled by the ratio between $\gamma$ and $\delta$, as $\| \boldsymbol{b} \rvert_{\mathcal{R}(A)} \| = 8 \gamma$ and $\| \boldsymbol{b} \rvert_{\mathcal{N}(A^\mathsf{T})} \| = 8 \delta$ hold.
Since $\mathcal{R}(A) = \mathcal{R}(A^\mathsf{T})$, GMRES should determine the least squares solution of $\min_{\boldsymbol{x} \in \mathbb{R}^n} \| \boldsymbol{b} - A \boldsymbol{x} \|$ for all $\boldsymbol{b} \in \mathbb{R}^{128}$ (\cref{th:GMRESconv_BW}).
Throughout all our numerical experiments, we use GMRES and RR-GMRES with the Householder orthogonalization process \cite{Walker1988} to ensure the best possible orthogonality among the Arnoldi basis vectors $\boldsymbol{q}_1$, $\boldsymbol{q}_2$, \dots, $\boldsymbol{q}_k$ and we compute the $k$th residual $\boldsymbol{r}_k = \boldsymbol{b} - A \boldsymbol{x}_k$ explicitly from $\boldsymbol{x}_k$ by solving the extended Hessenberg least squares problem $\min_{\boldsymbol{y} \in \mathbb{R}^k} \| \beta \boldsymbol{e}_1 - H_{k+1, k} \boldsymbol{y} \|$ with the Matlab backslash solver, which utilizes column pivoting.
A mathematically equivalent solution of $\min_{\boldsymbol{y} \in \mathbb{R}^k} \| \beta \boldsymbol{e}_1 - H_{k+1, k} \boldsymbol{y} \|$ was presented in \cite{LiaoHayamiYin2016}.

\cref{fig:relres} shows the relative residual norm $\| A^\mathsf{T} \boldsymbol{r}_k \| / \| A^\mathsf{T} \boldsymbol{b} \|$ versus the number of iterations of GMRES in the weakly inconsistent cases $(\gamma, \delta) = (1, 0)$, $ (1, 10^{-12})$, $(1, 10^{-8})$, and $(1, 10^{-4})$ on the left, and in the strongly inconsistent cases $(\gamma, \delta) = (1, 1)$, $(10^{-4}, 1)$, $(10^{-8}, 1)$, and $(10^{-12}, 1)$ on the right.
Note that the relative residual norm $\| A^\mathsf{T} \boldsymbol{r}_k \| / \| A^\mathsf{T} \boldsymbol{b} \|$ is associated with the normal equations $A^\mathsf{T} \! A \boldsymbol{x} = A^\mathsf{T} \boldsymbol{b}$, which are mathematically equivalent to the linear least squares problem $\min_{\boldsymbol{x} \in \mathbb{R}^n} \| \boldsymbol{b} - A \boldsymbol{x} \|$.
Similarly, \cref{fig:reserr,fig:sv} show the relative residual error norm $\| \boldsymbol{r}_k - \boldsymbol{r}_* \| / \| \boldsymbol{r}_k \|$ (cf.~equation \cref{eq:ubArr}) and the extremal singular values of $A$ and $H_{k+1, k}$, respectively.
If the inconsistency is small ($ \delta \ll \gamma$), then GMRES is sufficiently accurate (\cref{fig:relres_W}); otherwise the relative residual norm $\| A^\mathsf{T} \boldsymbol{r}_k \| / \| A^\mathsf{T} \boldsymbol{b} \|$ stagnates before attaining the accuracy on the level of $u \kappa (A)$ (\cref{fig:relres_S}), where $u \simeq 1.1 \cdot 10^{-16}$ is the unit roundoff.
In contrast to the nonsingular case, GMRES deteriorates not only due to  the condition number of $A$ but also due to the inconsistency measured here by $\delta > 0$.
For strongly inconsistent systems with $\delta \gg \gamma$, i.e., for $\boldsymbol{r}_0$ close to $\mathcal{N}(A)$, even though $\| A^\mathsf{T} \boldsymbol{r}_k \| / \| A^\mathsf{T} \boldsymbol{b} \|$ is large and stagnates, and $\| \boldsymbol{r}_* \|$ and hence $\| \boldsymbol{r}_k \|$ are large, the residual $\boldsymbol{r}_k$ approaches $\boldsymbol{r}_*$.
\cref{fig:sv_S} shows that for strongly inconsistent problems, $H_{k+1, k}$ has a condition number significantly larger than $A$, tends to become more ill-conditioned in the subsequent steps, and becomes numerically rank-deficient with $u \| H_{k+1, k} \| \| H_{k+1, k}^\dag \| \geq 1$ as the iteration proceeds.
In particular, for $\gamma=0$ and $\delta = 1$, GMRES breaks down at step $1$ but gives a least squares solution.
Comparing \cref{fig:reserr_S,fig:sv_S}, we see that in these cases the bound \cref{eq:ubArr} gives a reasonably good upper estimate for the smallest singular value of $H_{k+1, k}$.
The behavior described by this example is illustrated on a practical problem from a discretization of a partial differential equation with periodic boundary conditions in \cite[Experiment 4.2]{BrownWalker1997}.

A remedy for the ill-conditioning occurring in GMRES due to inconsistency is to form the Krylov subspace $\mathcal{K}_k (A, A \boldsymbol{r}_0)$ by starting with the initial vector $A \boldsymbol{r}_0$ in $\mathcal{R}(A)$ instead of $\boldsymbol{r}_0$ as is done in RR-GMRES.
Note that, on the other hand,  the RR-GMRES residual norm is always larger than or equal to the GMRES residual norm (see \cref{sec:GMRsing}).
Similarly to the above, we show numerical results for RR-GMRES on the same inconsistent linear systems \cref{eq:test_inc}.
\Cref{fig:RRrelres,fig:RRreserr,fig:RRsv} show the same quantities as \cref{fig:relres,fig:reserr,fig:sv} for RR-GMRES.
For any inconsistency parameter $\delta > 0$, the condition number of $H_{k+1, k}^R$ is bounded above by the condition number of $A$ and RR-GMRES is sufficiently accurate, as
\begin{align*}
\sigma_k (H_{k+1, k}^\mathrm{R}) & = \min_{\boldsymbol{z} \in \mathcal{K}_k (A, A \boldsymbol{r}_0) \backslash \lbrace \boldsymbol{0} \rbrace} \frac{\| A \boldsymbol{z} \|}{\| \boldsymbol{z} \|} \\
& \geq \min_{\boldsymbol{z} \in \mathcal{R}(A) \backslash \lbrace \boldsymbol{0} \rbrace} \frac{\| A \boldsymbol{z} \|}{\| \boldsymbol{z} \|} = \min_{\boldsymbol{z} \in \mathcal{R}(A^\mathsf{T}) \backslash \lbrace \boldsymbol{0} \rbrace} \frac{\| A \boldsymbol{z} \|}{\| \boldsymbol{z} \|} = \sigma_r (A) 
\end{align*}
for $\mathcal{R}(A) = \mathcal{R}(A^\mathsf{T})$.
Thus, the accuracy of the RR-GMRES iterate is affected 
\begin{figure}[ht!]
	\centering
	\subfloat[Weakly inconsistent cases ($\gamma = 1$). \label{fig:relres_W}]{\includegraphics[width=180pt]{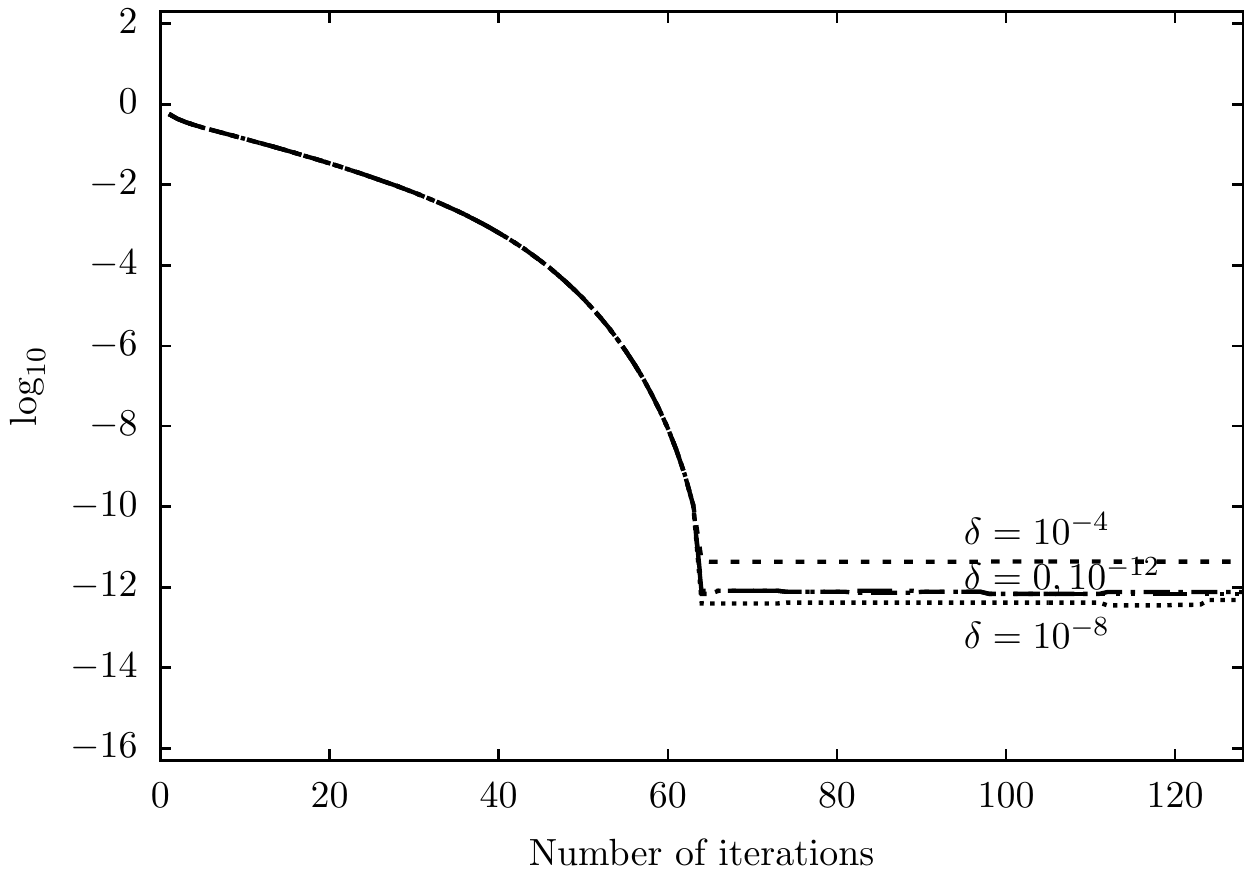}}
	\hspace{6pt}
	\subfloat[Strongly inconsistent cases ($\delta = 1$). \label{fig:relres_S}]{\includegraphics[width=180pt]{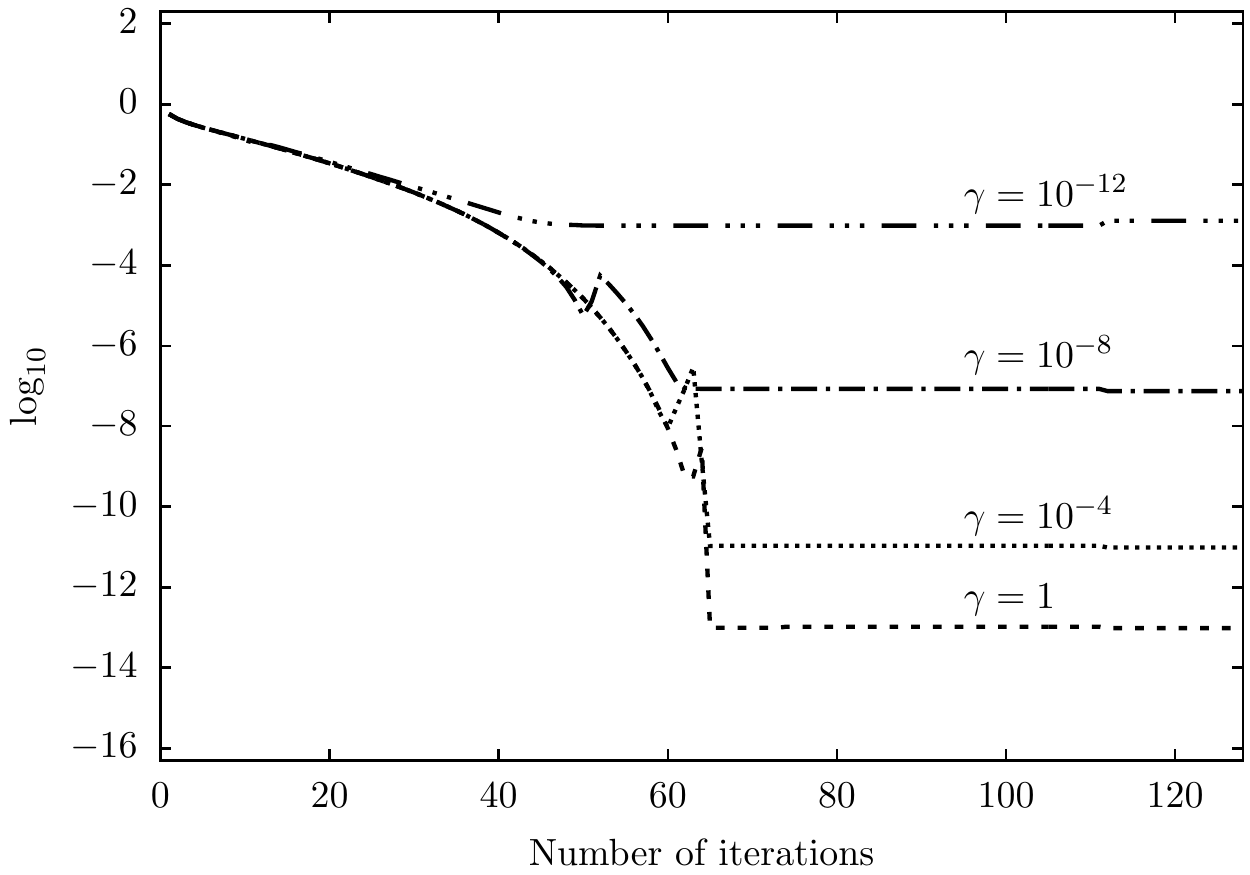}}
	\caption{Relative residual norm $\| A^\mathsf{T} \boldsymbol{r}_k \| / \| A^\mathsf{T} \boldsymbol{b} \|$ for GMRES.}
	\label{fig:relres}
	\centering
	\subfloat[Weakly inconsistent cases ($\gamma = 1$). \label{fig:reserr_W}]{\includegraphics[width=180pt]{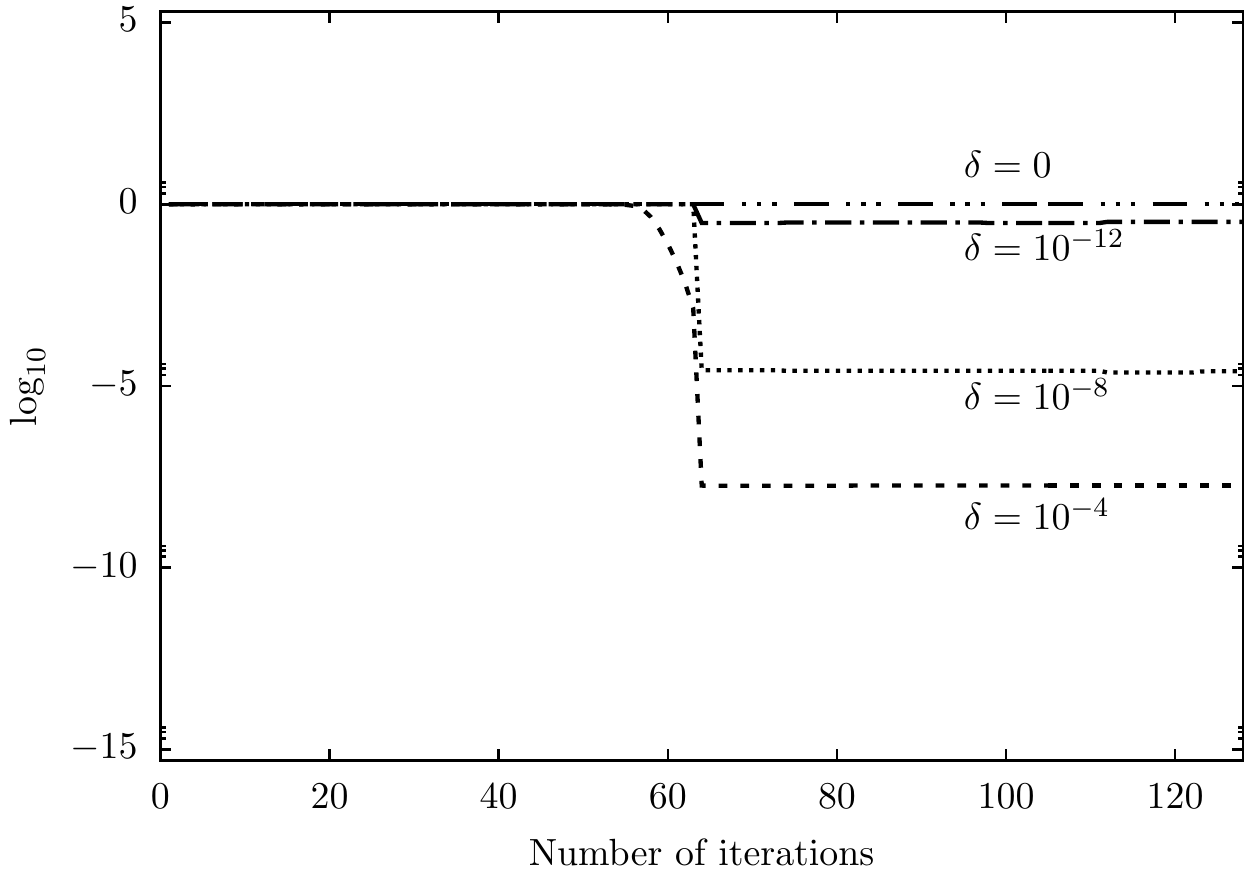}}
	\hspace{6pt}
	\subfloat[Strongly inconsistent cases ($\delta = 1$). \label{fig:reserr_S}]{\includegraphics[width=180pt]{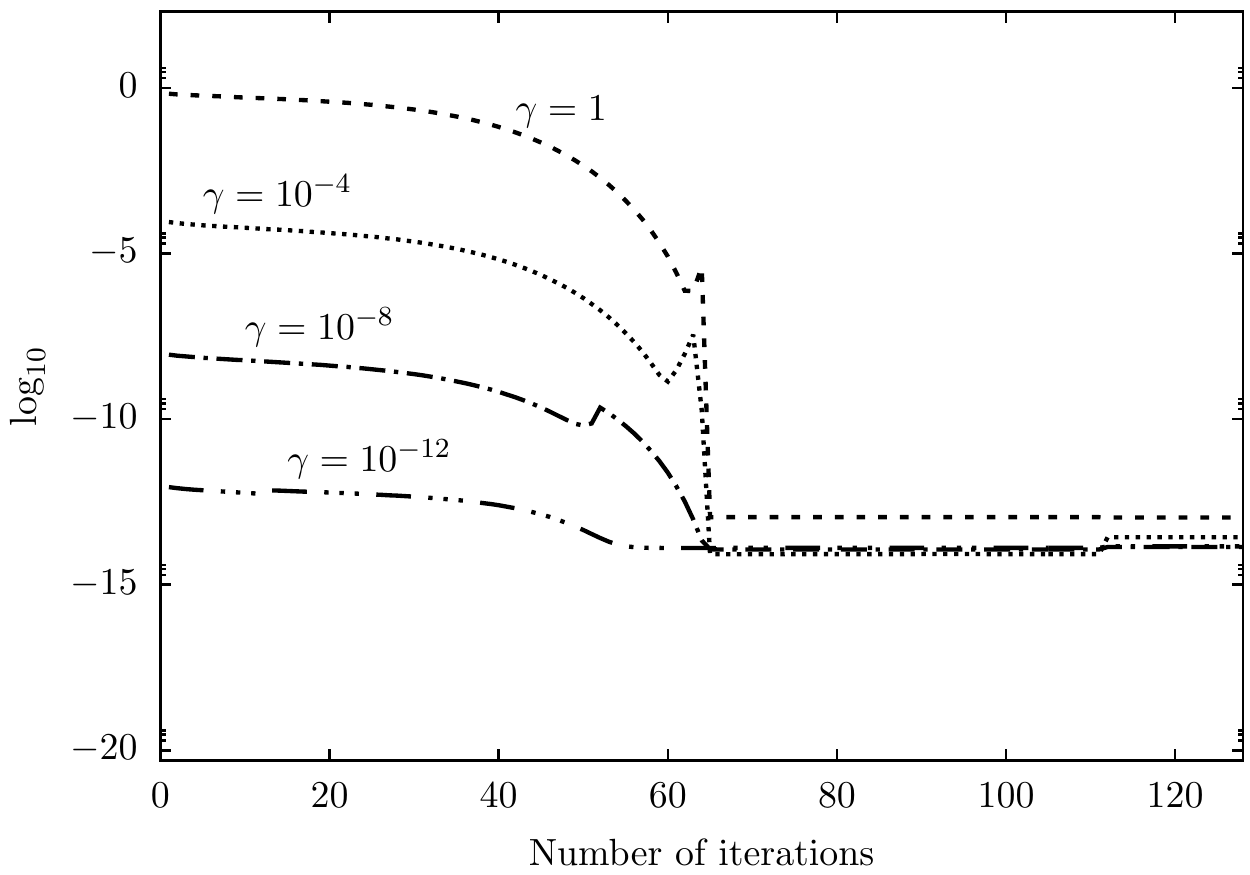}}
	\caption{Relative residual error norm $\| \boldsymbol{r}_k - \boldsymbol{r}_* \| / \| \boldsymbol{r}_k \|$ for GMRES.}
	\label{fig:reserr}
	\subfloat[Weakly inconsistent cases ($\gamma = 1$).]{\includegraphics[width=180pt]{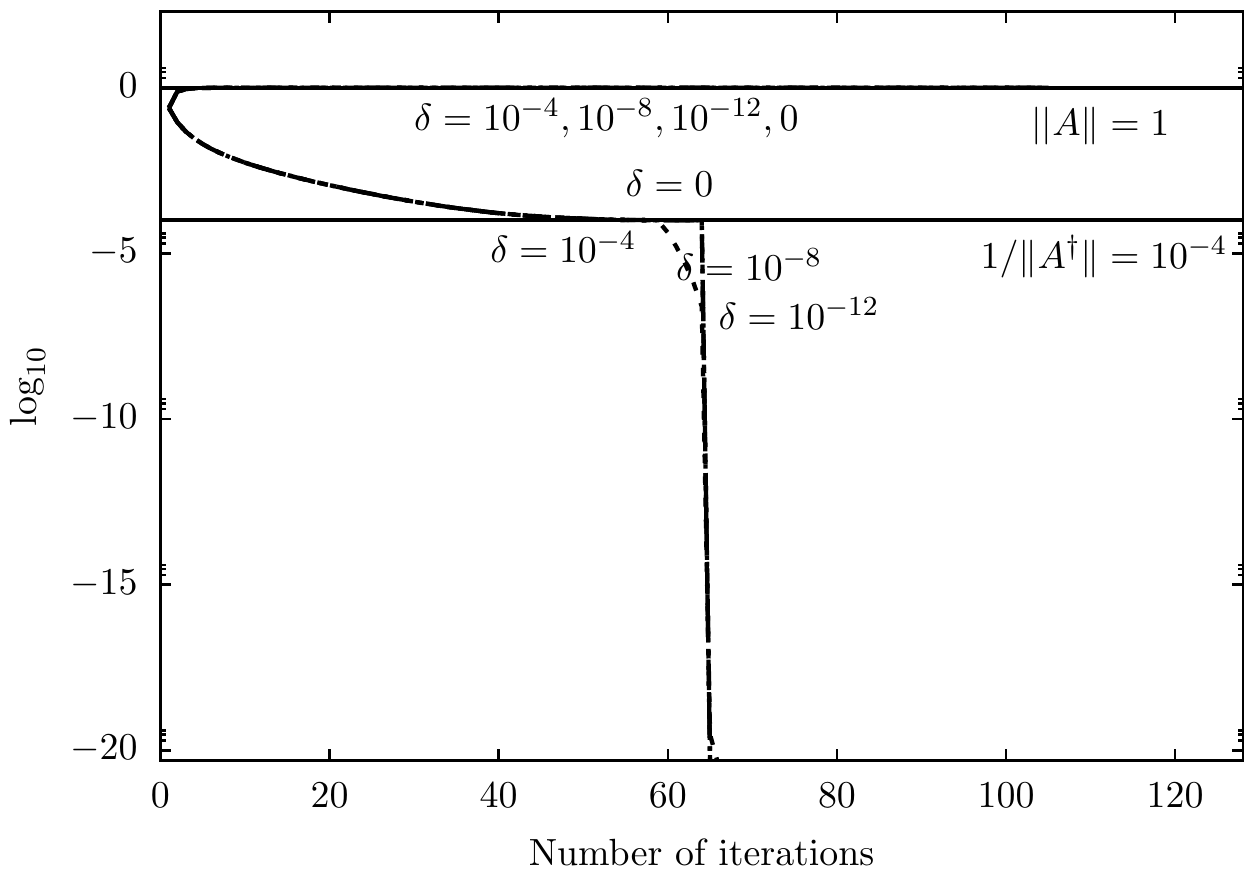}}
	\hspace{6pt}
	\subfloat[Strongly inconsistent cases ($\delta = 1$). \label{fig:sv_S}]{\includegraphics[width=180pt]{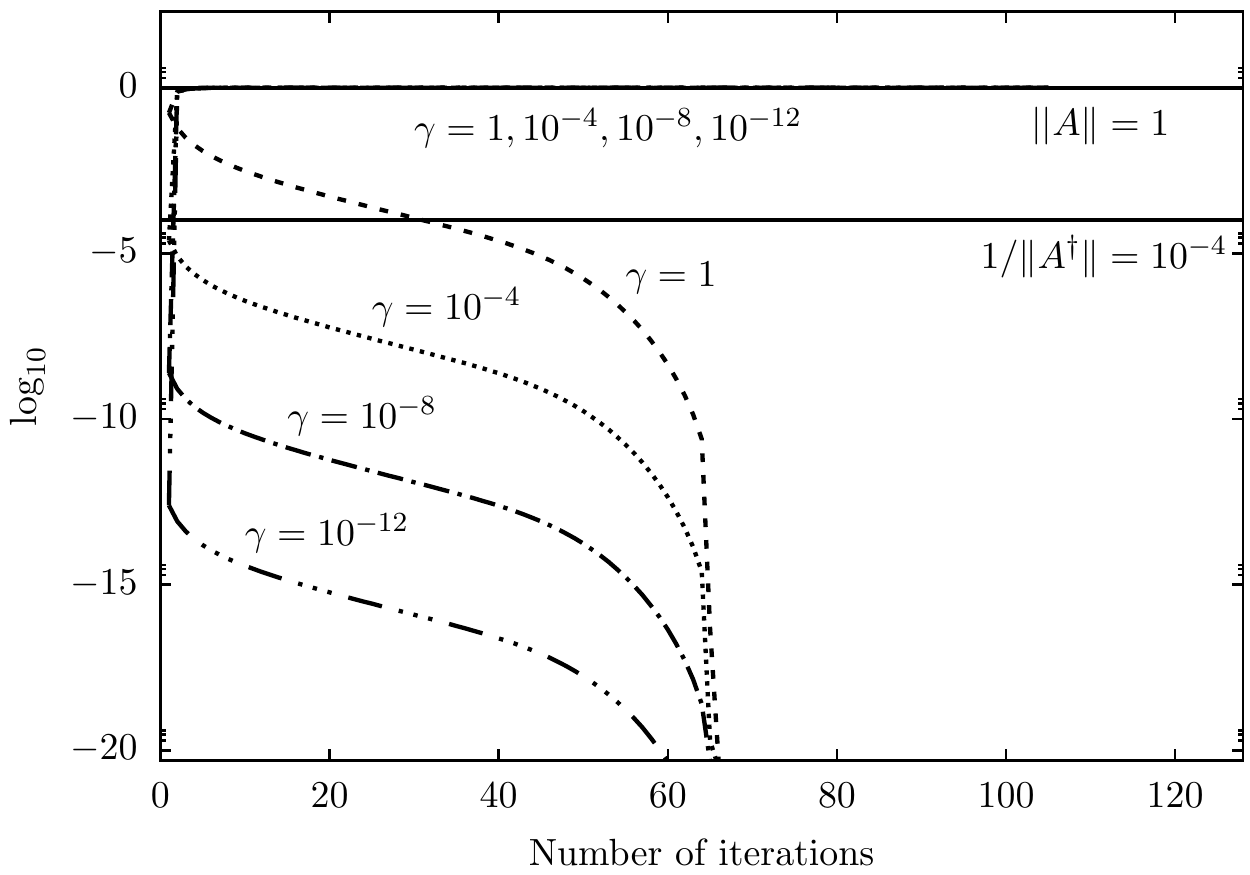}}
	\caption{Extremal singular values of $A$ and $H_{k+1, k}$ for GMRES.}
	\label{fig:sv}
\end{figure}
only by the condition number of $A$, even though the inconsistency increases or $\boldsymbol{r}_0$ approaches $\mathcal{N}(A)$.
Hence, for inconsistent problems  with EP matrices RR-GMRES is a successful alternative to GMRES.

\begin{figure}[ht!]
	\centering
	\subfloat[Weakly inconsistent cases ($\gamma = 1$). \label{fig:W_relres_RR}]{\includegraphics[width=180pt]{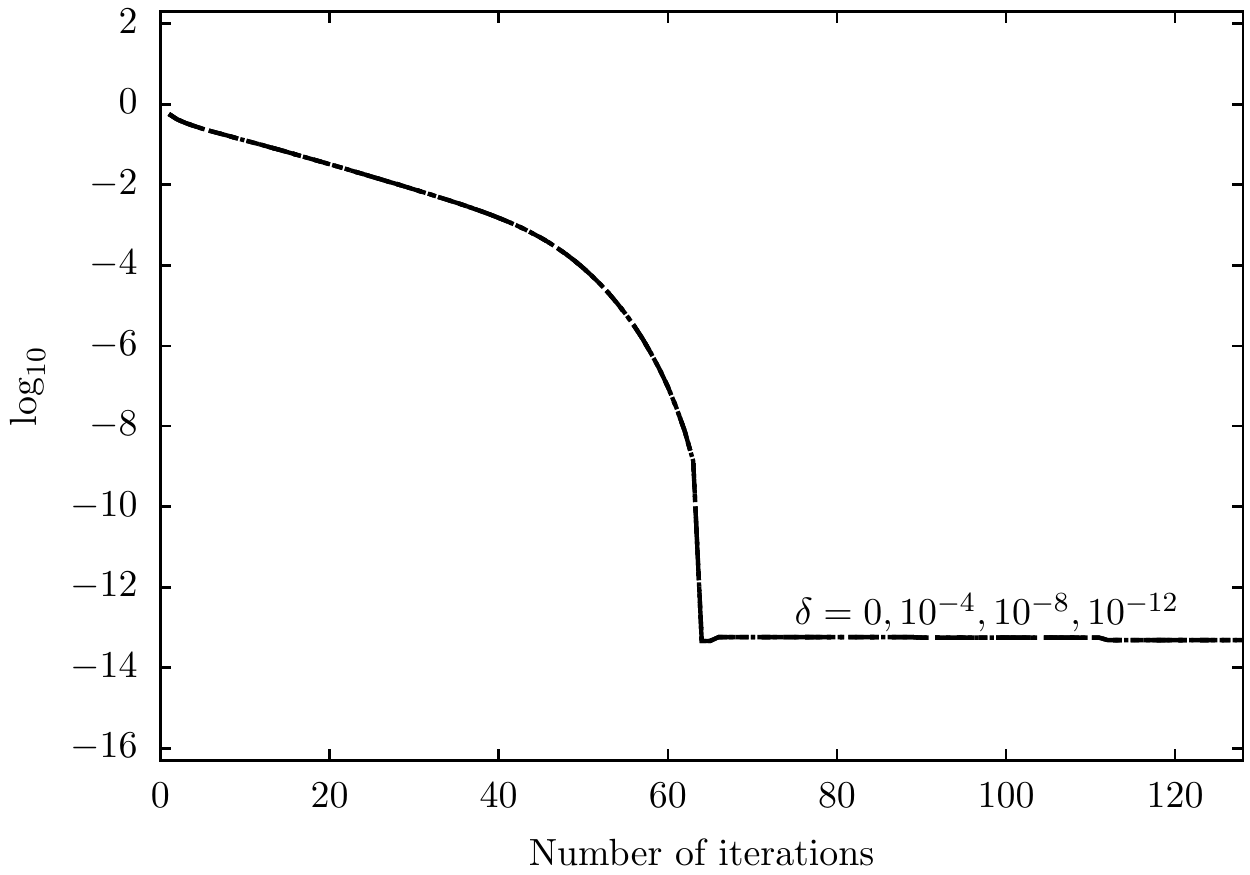}}
	\hspace{6pt}
	\subfloat[Strongly inconsistent cases ($\delta = 1$). \label{fig:S_relres_RR}]{\includegraphics[width=180pt]{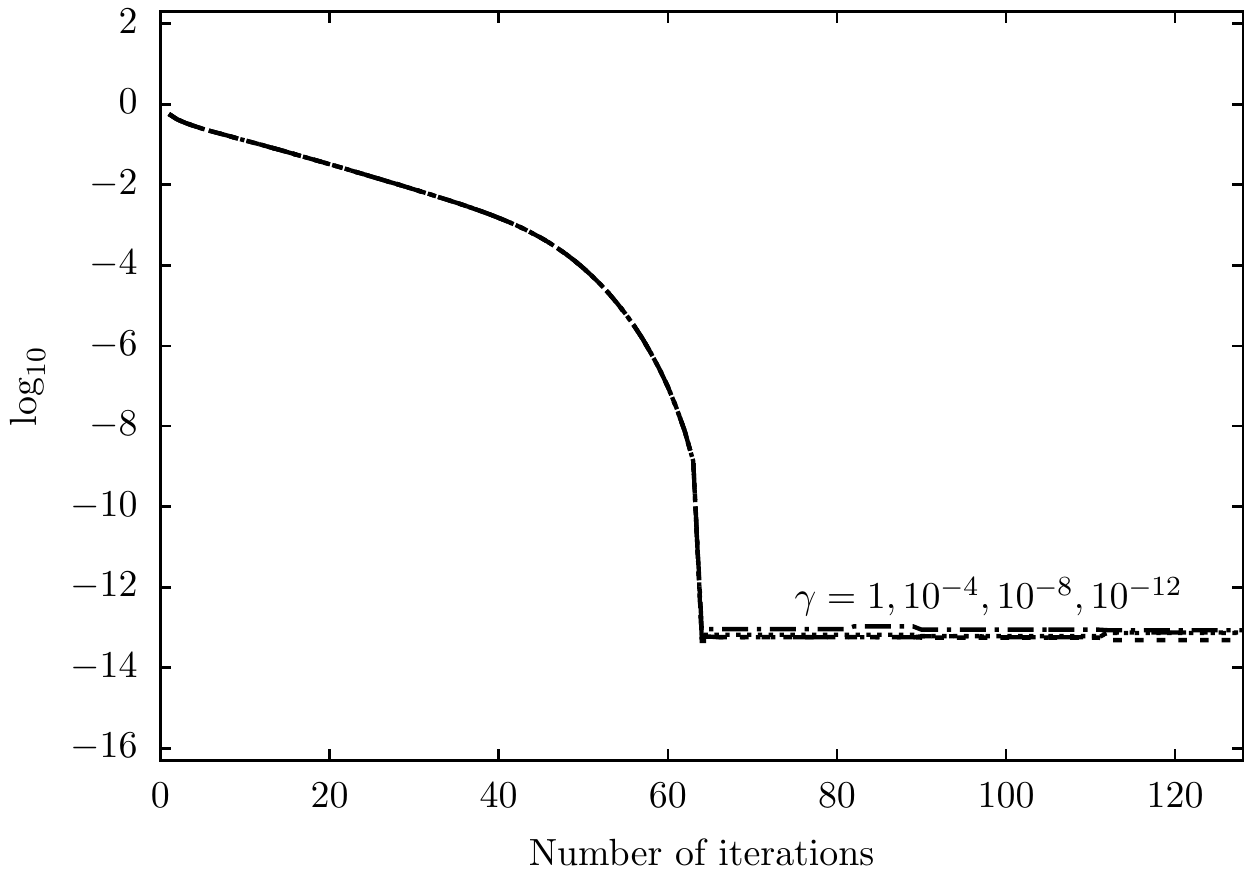}}
	\caption{Relative residual norm $\| A^\mathsf{T} \boldsymbol{r}_k^\mathrm{R} \| / \| A^\mathsf{T} \boldsymbol{b} \|$ for RR-GMRES.}
	\label{fig:RRrelres}
	\centering
	\subfloat[Weakly inconsistent cases ($\gamma = 1$). \label{fig:W_reserr_RR}]{\includegraphics[width=180pt]{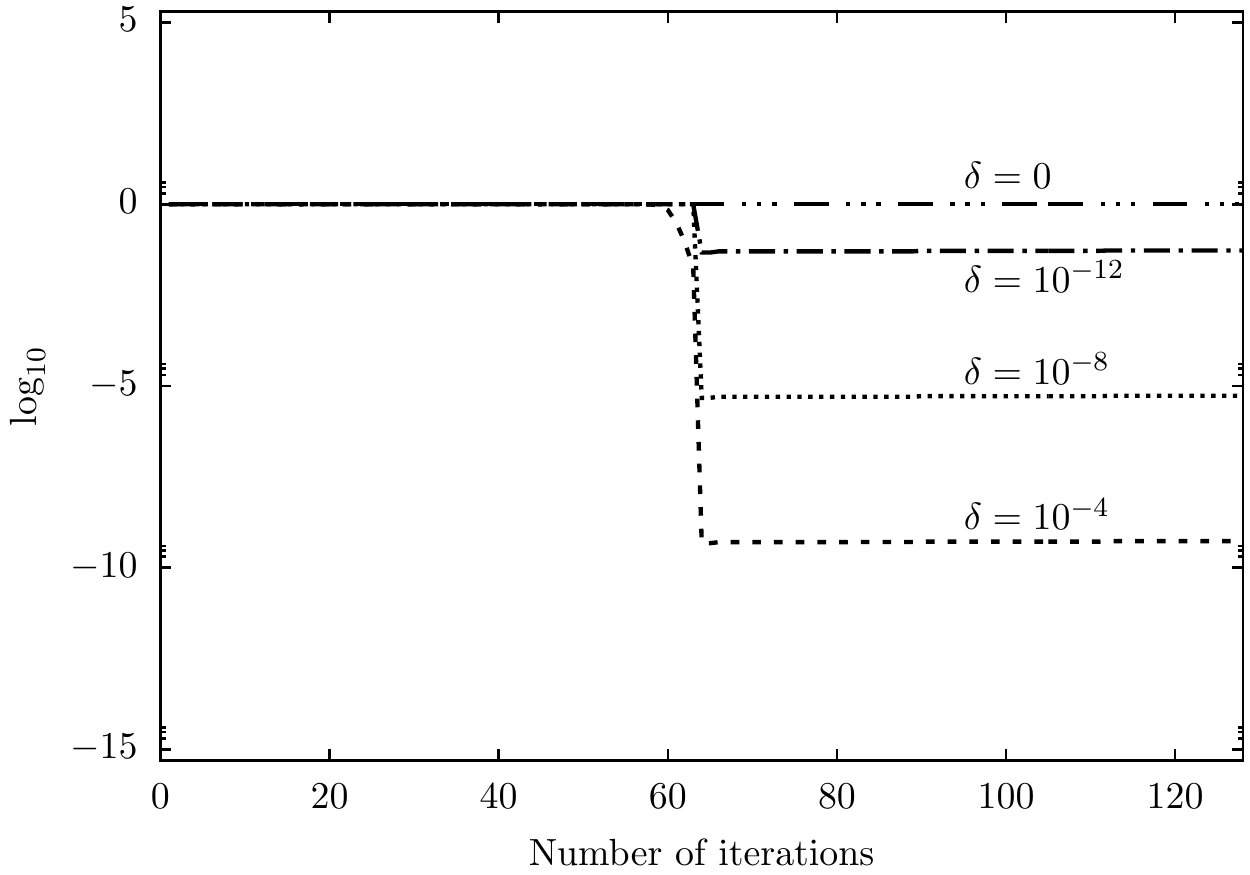}}
	\hspace{6pt}
	\subfloat[Strongly inconsistent cases ($\delta = 1$). \label{fig:S_reserr_RR}]{\includegraphics[width=180pt]{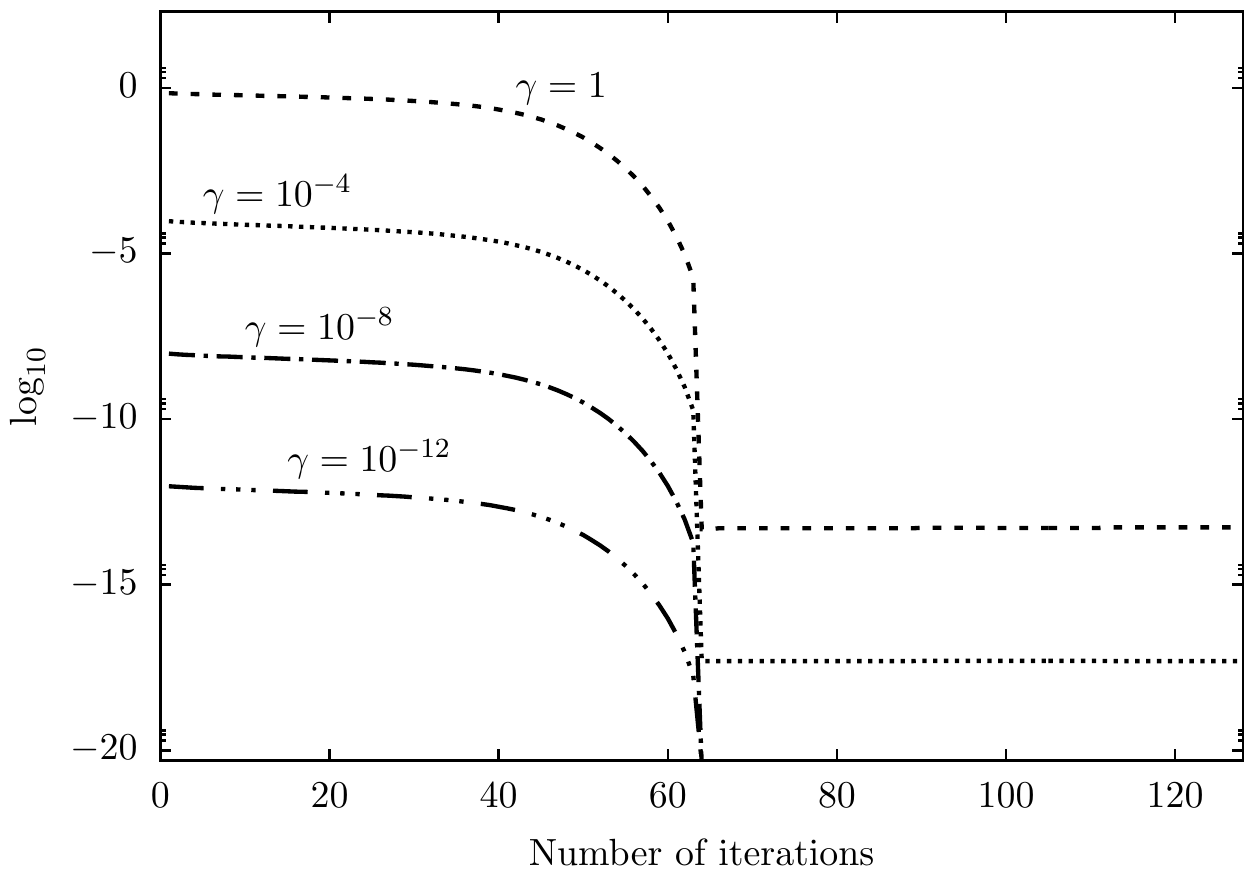}}
	\caption{Relative residual error norm $\| \boldsymbol{r}_k^\mathrm{R} - \boldsymbol{r}_* \| / \| \boldsymbol{r}_k^\mathrm{R} \|$ for RR-GMRES.}
	\label{fig:RRreserr}
	\subfloat[Weakly inconsistent cases ($\gamma = 1$).]{\includegraphics[width=180pt]{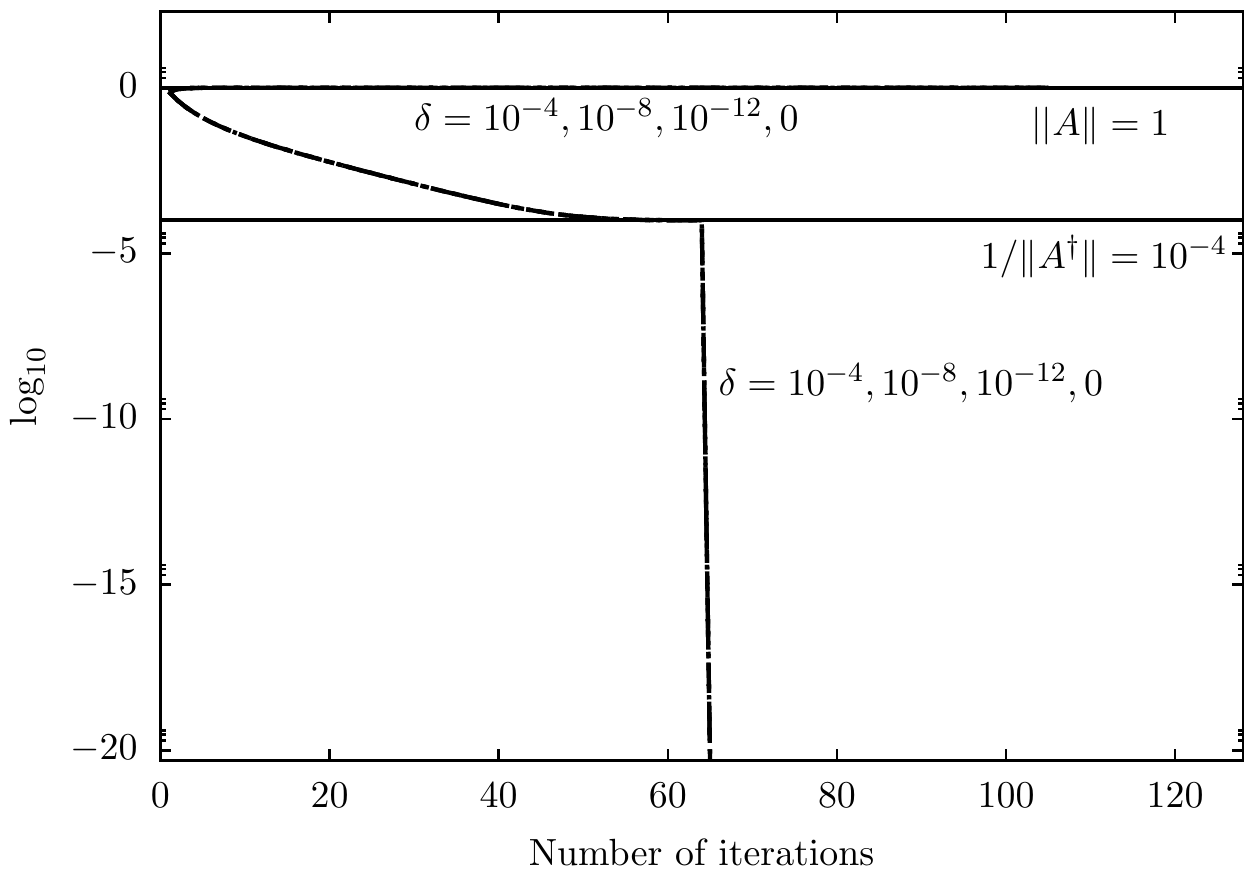}}
	\hspace{6pt}
	\subfloat[Strongly inconsistent cases ($\delta = 1$).]{\includegraphics[width=180pt]{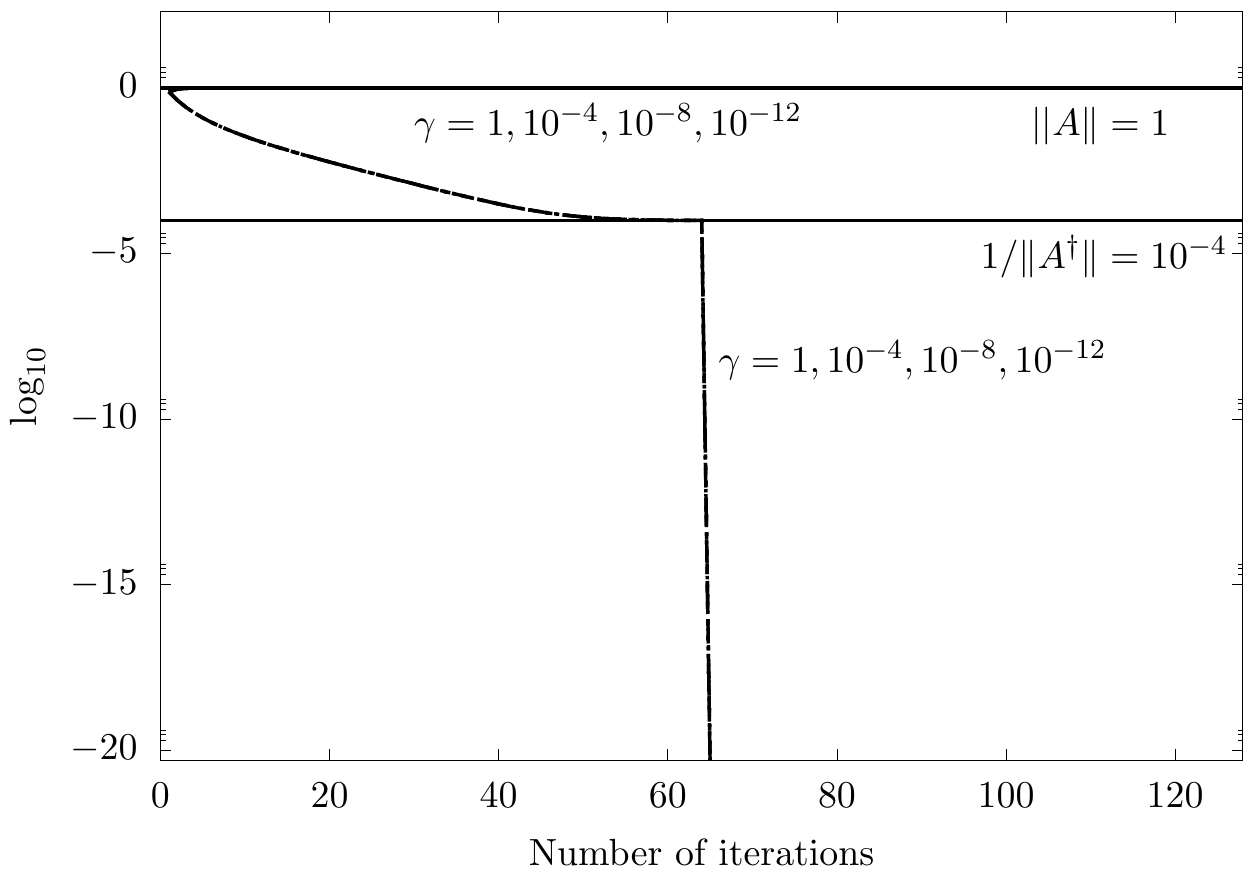}}
	\caption{Extremal singular values of $A$ and $H_{k+1, k}^\mathrm{R}$ for RR-GMRES.}
	\label{fig:RRsv}
\end{figure}

\section{GMRES and GP matrices}\label{sec:GMR_GP}

We have shown in \cref{sec:GMR_EP} that the condition number $\kappa(A\rvert_{\mathcal{R}(A)})$ plays an important role in the behavior of GMRES and in the EP case we have $\kappa(A\rvert_{\mathcal{R}(A)})= \kappa(A)$.
Thus for consistent problems with EP matrices, the condition number of $A$ represents an upper bound for the condition number of $H_{k+1, k}$ due to $\kappa(H_{k+1, k}) \leq \kappa(A\rvert_{\mathcal{R}(A)})=\kappa(A)$ and the accuracy of the GMRES iterates is actually determined by the singular values of $A$.
Consider now applying GMRES to $A \boldsymbol{x} = \boldsymbol{b}$, where $A$ is a GP matrix (\cref{th:GMRESconv_SS}).
We will show that in the GP case $\kappa(A\rvert_{\mathcal{R}(A)})$ can be significantly larger than $\kappa(A)$ and thus the condition number $\kappa(H_{k+1, k})$ can become larger than $\kappa(A)$ even in the consistent case. 
The accuracy of the GMRES iterates can be then affected by the inaccurate solution of the extended Hessenberg least squares problem that can be ill-conditioned even if $A$  is well-conditioned.

According to \cref{th:GMRESconv_SS}, GMRES in the consistent GP case determines $\boldsymbol{x}_\# + (\mathrm{I} - A^\# \! A) \boldsymbol{x}_0$.
The vector $\boldsymbol{x}_\#$ that belongs to $\mathcal{R}(A)$ can be related to the vector $\boldsymbol{x}_*$ that belongs to $\mathcal{R}(A^\mathsf{T})$ as follows:
\begin{align}
\sigma_r(V_1^\mathsf{T} U_1) \| \boldsymbol{x}_\# \| \leq \| \boldsymbol{x}_* \| \leq \| \boldsymbol{x}_\# \|,
\label{eq:ineq_xastxsharp}
\end{align}
which follows from the identity $\boldsymbol{x}_* = P_{\mathcal{R}(A^\mathsf{T})} \boldsymbol{x}_\# =  V_1 V_1^\mathsf{T} U_1 U_1^\mathsf{T} \boldsymbol{x}_\#$.
Note that $\boldsymbol{x}_\#$ may have a large component in $\mathcal{N}(A)$, if the angle between $\mathcal{N}(A)$ and $\mathcal{R}(A)$ is small, which may affect the accuracy of GMRES iterates (see \cref{fig:sol_GP}), due to the ill-conditioning of the extended Hessenberg matrix $H_{k+1, k}$.

In the consistent case, the extremal singular values of $H_{k+1,k}$ can be bounded as
\begin{align*}
\sigma_1(H_{k+1,k}) \leq \max_{\boldsymbol{z} \in \mathcal{R}(A) \backslash \lbrace \boldsymbol{0} \rbrace} \frac{\| A \boldsymbol{z} \|}{\| \boldsymbol{z} \|} = \max_{\boldsymbol{z} \in \mathbb{R}^r \backslash \lbrace \boldsymbol{0} \rbrace} \frac{\| U_1 \Sigma_r V_1^\mathsf{T} U_1 \boldsymbol{z} \|}{\| U_1 \boldsymbol{z} \|} 	\leq \| A \| \| V_1^\mathsf{T} U_1 \|,
\end{align*}
\begin{align}
\sigma_k(H_{k+1,k}) \geq \min_{\boldsymbol{z} \in \mathcal{R}(A) \backslash \lbrace \boldsymbol{0} \rbrace} \frac{\| A \boldsymbol{z} \|}{\| \boldsymbol{z} \|} = \min_{\boldsymbol{z} \in \mathbb{R}^r \backslash \lbrace \boldsymbol{0} \rbrace} \frac{\| U_1 \Sigma_r V_1^\mathsf{T} U_1 \boldsymbol{z} \|}{\| U_1 \boldsymbol{z} \|} \geq \sigma_r (A) \sigma_r (V_1^\mathsf{T} U_1).
\label{eq:lbHconGP}
\end{align}
Consequently, $\kappa(H_{k+1,k}) \leq \kappa(A)  \kappa (V_1^\mathsf{T} U_1)$ is related to the extremal principal angles between $\mathcal{R}(A)$ and $\mathcal{R}(A^\mathsf{T})$ (cf.\ \cite[Theorem~2.1]{Wei2000}).
The lower bound \eqref{eq:lbHconGP} shows that, in the consistent case, the smallest singular value of $H_{k+1, k}$ can be smaller than the smallest nonzero singular value of $A$, depending on the smallest nonzero singular value  of $V_1^\mathsf{T} U_1$.
In addition, it is easy to see that $\sigma_k (H_{k+1,k}) $ can be bounded by
\begin{align}
\sigma_k (H_{k+1, k}) = \min_{\boldsymbol{z} \in \mathcal{K}_k \backslash \lbrace \boldsymbol{0} \rbrace} \frac{\| A \boldsymbol{z} \|}{\| \boldsymbol{z} \|} \leq \frac{\| A \boldsymbol{r}_0 \|}{\| \boldsymbol{r}_0 \|} \leq \frac{\| A \| \| \boldsymbol{r}_0 \rvert_{\mathcal{R}(A^\mathsf{T})} \|}{\| \boldsymbol{r}_0 \|}.
\label{eq:uppbound_sgmk_H}
\end{align}
Here, the last inequality is implied by the splitting $A \boldsymbol{r}_0 = A (\boldsymbol{r}_0 \rvert_{\mathcal{N}(A)} + \boldsymbol{r}_0 \rvert_{\mathcal{R}(A^\mathsf{T})}) = A \boldsymbol{r}_0 \rvert_{\mathcal{R}(A^\mathsf{T})}$.
Although Brown and Walker mention in \cite[p.~50]{BrownWalker1997} that the condition number of $A \rvert_{\mathcal{K}_k}$ cannot become arbitrarily large through an unfortunate 
\begin{figure}[b]
	\centering
	\includegraphics[width=160pt]{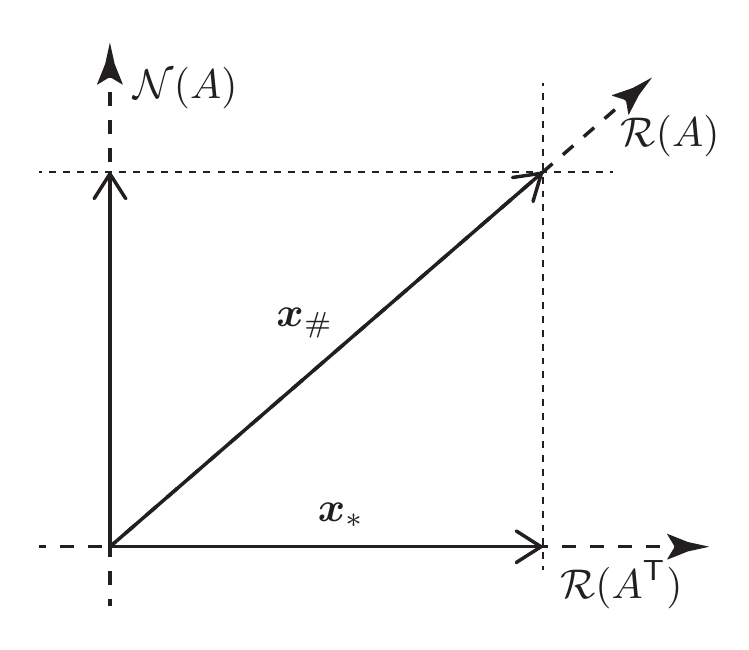}
	\caption{Geometric illustration of solution vectors in the GP case.}
	\label{fig:sol_GP}
\end{figure}
choice of $\boldsymbol{b}$ and $\boldsymbol{x}_0$, it is clear from \cref{eq:uppbound_sgmk_H} that if the residual $\boldsymbol{r}_0$ has a very small component in $\mathcal{R}(A^\mathsf{T})$ then the condition number of $H_{k+1,k}$ can be large for any singular matrix $A$.

In the following, we give illustrative examples that lead to ill-conditioned extended Hessenberg matrix $H_{k+1, k}$ in  GMRES.
First, we consider GMRES with $\boldsymbol{x}_0 = \boldsymbol{0}$ applied to $A \boldsymbol{x} = \boldsymbol{b}$, where
\begin{align}
A =
\begin{bmatrix}
\varepsilon & 1 \\
0 & 0
\end{bmatrix}, \quad
A^\# =
\begin{bmatrix}
1/\varepsilon & 1/\varepsilon^2 \\
0 & 0
\end{bmatrix}, \quad
\boldsymbol{b} =
\begin{bmatrix}
1 \\
0
\end{bmatrix}.
\label{eq:2x2GP}
\end{align}
The matrix $A$ has the following ranges and nullspaces
\begin{align*}
\mathcal{R}(A) =
\mathrm{span}
\left\lbrace
\begin{bmatrix}
1 \\
0
\end{bmatrix} \right\rbrace, \quad
\mathcal{N}(A) =
\mathrm{span}
\left\lbrace
\begin{bmatrix}
1 \\
- \varepsilon
\end{bmatrix}
\right\rbrace, \\
\mathcal{R}(A^\mathsf{T}) =
\mathrm{span}
\left\lbrace
\begin{bmatrix}
\varepsilon \\
1
\end{bmatrix}
\right\rbrace, \quad
\mathcal{N}(A^\mathsf{T}) =
\mathrm{span}
\left\lbrace
\begin{bmatrix}
0 \\
1
\end{bmatrix}
\right\rbrace.
\end{align*}
It is clear that for $\varepsilon=0$ the matrix $A$ is DR (see \cref{sec:GMRsing}).
In addition, the initial vector $\boldsymbol{r}_0$ satisfies $\boldsymbol{r}_0 \in \mathcal{N}(A) \cap \mathcal{R}(A)$ and thus the GMRES method breaks down at step 1.
Now suppose that $0 < \varepsilon \ll 1$.
Then, the matrix $A$ is GP but not EP, i.e., $\mathcal{R}(A^\mathsf{T}) \ne\mathcal{R}(A)$. 
Since $\sigma_1 (V_1^\mathsf{T} U_1)= \varepsilon / \sqrt{1 + \varepsilon^2}$, we have 
\begin{align*}
\min_{\boldsymbol{z} \in \mathcal{R}(A) \backslash \lbrace \boldsymbol{0} \rbrace} \frac{\| A \boldsymbol{z} \|}{\| \boldsymbol{z} \|} = \varepsilon, \quad
\min_{\boldsymbol{z} \in \mathcal{R}(A^\mathsf{T}) \backslash \lbrace \boldsymbol{0} \rbrace} \frac{\| A \boldsymbol{z} \|}{\| \boldsymbol{z} \|} = \sqrt{1 + \varepsilon^2}.
\end{align*}
The smallest singular value of $H_{2, 1}$ is  significantly smaller than the smallest nonzero singular value of $A$, $\sigma_1 (H_{2, 1}) = \varepsilon \ll \sqrt{1 + \varepsilon^2} = \sigma_1 (A)$.
Indeed, the components of $H_{2, 1}$ are $H_{2, 1} = [\varepsilon, 0]^\mathsf{T}$.
Furthermore, by solving $\min_{\boldsymbol{y} \in \mathbb{R}^1} \| \beta \boldsymbol{e}_1 - H_{2, 1} \boldsymbol{y} \|$ with $\beta = 1$,
we have $\boldsymbol{y}_1 = 1 / \varepsilon$, i.e., $\boldsymbol{y}_1$ has a large component.
We see that $\boldsymbol{x}_1 = Q_1 \boldsymbol{y}_1$  and  $\boldsymbol{x}_1 = \boldsymbol{x}_\# = A^\# \boldsymbol{b} = [1/\varepsilon, 0]^\mathsf{T}$ for $\boldsymbol{b} = [1, 0]^\mathsf{T}$.
Therefore, $\|  \boldsymbol{x}_1 \|$ becomes very large even if the condition number of $A$ and the norm of the right-hand side are small.
Thus, the vector $\boldsymbol{x}_\#$ contains a large component in $\mathcal{N}(A)$, whereas $\boldsymbol{x}_* = 1 / (1+\varepsilon^2) [\varepsilon, 1]^\mathsf{T}$, and the inequalities \eqref{eq:ineq_xastxsharp} are satisfied.


In the following, we also consider the nonsingular ill-posed linear system $A \boldsymbol{x} = \boldsymbol{b}$, where 
\begin{align}
A = 
\begin{bmatrix}
\varepsilon & 1 \\
0 & \delta
\end{bmatrix},
\quad
\boldsymbol{b} =
\begin{bmatrix}
1 \\
-\varepsilon
\end{bmatrix}
\quad \delta, \varepsilon > 0,
\label{eq:2x2nearGP}
\end{align}
where $\delta$ is a small parameter and the right-hand side $\boldsymbol{b}$ is contaminated by the error $[0, -\varepsilon]^\mathsf{T}$.
The extremal singular values of $A$ are $\sigma_1 (A) \simeq 1$ and $\sigma_2 (A) \simeq \delta \varepsilon$.
Indeed, the condition number of $A$ is bounded as 
\begin{align}
\kappa (A) = & \left( \frac{1 + \delta^2 + \varepsilon^2 + \sqrt{(1 + \delta^2 + \varepsilon^2)^2 - 4 \delta^2 \varepsilon^2}}{1 + \delta^2 + \varepsilon^2 - \sqrt{(1 + \delta^2 + \varepsilon^2)^2 - 4 \delta^2 \varepsilon^2}} \right)^{1/2} \geq \frac{1}{\delta \varepsilon}
\end{align}
Then, GMRES applied to $A \boldsymbol{x} = \boldsymbol{b}$ with the error-free right-hand side $\boldsymbol{b} = [1, 0]^\mathsf{T}$ and $\boldsymbol{x}_0 = \boldsymbol{0}$ gives the same Arnoldi vector $\boldsymbol{q}_1$ and the same extended Hessenberg matrix $H_{2, 1}$ as the ones for \cref{eq:2x2GP}.
GMRES applied to $A \boldsymbol{x} = \boldsymbol{b}$ with \eqref{eq:2x2nearGP} give 
\begin{align*}
Q_2 = \frac{1}{\sqrt{1+\varepsilon^2}} 
\begin{bmatrix}
1 & \varepsilon \\
-\varepsilon & 1
\end{bmatrix},
\quad 
H_{2, 2} = \frac{1}{1+\varepsilon^2}
\begin{bmatrix}
\delta \varepsilon^2 & 1 + \varepsilon^2 - \delta \varepsilon\\
\delta \varepsilon & \varepsilon (1+\varepsilon^2) + \delta
\end{bmatrix}.
\end{align*}
Hence, we have the upper bound $\| A^{-1} \boldsymbol{b} \| = \| \boldsymbol{x}_2 \|=\| \boldsymbol{y}_2 \| \leq \| A^{-1} \|  \| \boldsymbol{b} \| = \| {H_{2, 2}}^{-1} \| \| \boldsymbol{r}_0 \| = \break \sqrt{1+\varepsilon^2} / (\delta \epsilon)$.
This means that GMRES breaks down in the exact singular case, whereas, in the ill-posed case, GMRES does not break down but the iterates computed by GMRES can be inaccurate as $\delta$, $\epsilon$ go to zero.
In both exactly singular case and nonsingular ill-posed case, the iterates computed by GMRES will be inaccurate due to the ill-conditioning of $H_{2, 2}$.

The above behavior of GMRES for singular linear systems with respect to the condition number of $V_1^\mathsf{T} U_1$ is illustrated on numerical examples $A \boldsymbol{x} = \boldsymbol{b}$, where
\begin{align}
A =
\begin{bmatrix}
D & \mathrm{I} \\
\mathrm{O} & \mathrm{O}
\end{bmatrix} \in \mathbb{R}^{128 \times 128}, \quad \boldsymbol{b} = \begin{bmatrix} \boldsymbol{f} \\ \boldsymbol{0}  \end{bmatrix},
\label{eq:DI}
\end{align}
$D = \mathrm{diag} (d_1, d_2, \dots, d_{64}) \in \mathbb{R}^{64 \times 64}$ is a diagonal matrix whose diagonal entries have the so-called Strako{\v s} distribution \cite{Strakos1991} 
\begin{align*}
d_1 = 1, \quad d_{64} = 10^{-\rho}, \quad d_i = d_{64} + \frac{64-i}{63}(d_1 - d_{64}) \cdot 0.7^{i-1}, \quad i = 2, 3, \dots, 63,
\end{align*}
and $\boldsymbol{f} = (f_i) \in \mathbb{R}^{64}$ has the entries $f_i = 10^{-(64 - i) \rho / 63}$, $j = 1, 2, \dots, 64$.
Note that $A$ is GP but not EP.
This setting gives well-conditioned $A$ with $\kappa (A) = \sqrt{2 / (10^{-2 \rho} + 1)} \simeq \sqrt{2}$ and ill-conditioned $V_1^\mathsf{T} U_1$ for $\kappa (V_1 ^\mathsf{T}U_1) = 10^\rho \sqrt{(10^{-2 \rho} + 1) / 2} \simeq 10^\rho / \sqrt{2}$ for $\rho \gg 1$.
Furthermore, the norms of vectors $\boldsymbol{f}$, $D \boldsymbol{f}$, $D^2 \boldsymbol{f}$ decrease and this reduction is pronounced as the value of $\rho$ increases.

\Cref{fig:DI_nrm_r_minsv_b=A[0d],fig:DI_RRnrm_r_minsv_b=A[0d]} show the relative residual norm and the smallest singular value of $H_{k+1, k}$ and $H_{k+1, k}^\mathrm{R}$ versus the number of iterations for GMRES and RR-GMRES, respectively, applied to the above linear systems with $\rho = 1$, $4$, $8$, and $12$.
As the value of $\rho$ increases, the condition number of the extended Hessenberg matrix increases, and the accuracy of the relative residual for both GMRES and RR-GMRES is significantly lost.
It is clear from our experiments that while RR-GMRES does help in the inconsistent EP case by starting with a vector in $\mathcal{R}(A)$ to construct the Krylov subspace, in the GP case both GMRES and RR-GMRES may not give accurate solutions, even in the consistent case, when the condition number of $V_1^\mathsf{T} U_1$ is large.
\begin{figure}[b!]
	\centering
	\subfloat[Relative residual norm $\| \boldsymbol{r}_k \| / \| \boldsymbol{b} \|$.]{\includegraphics[width=180pt]{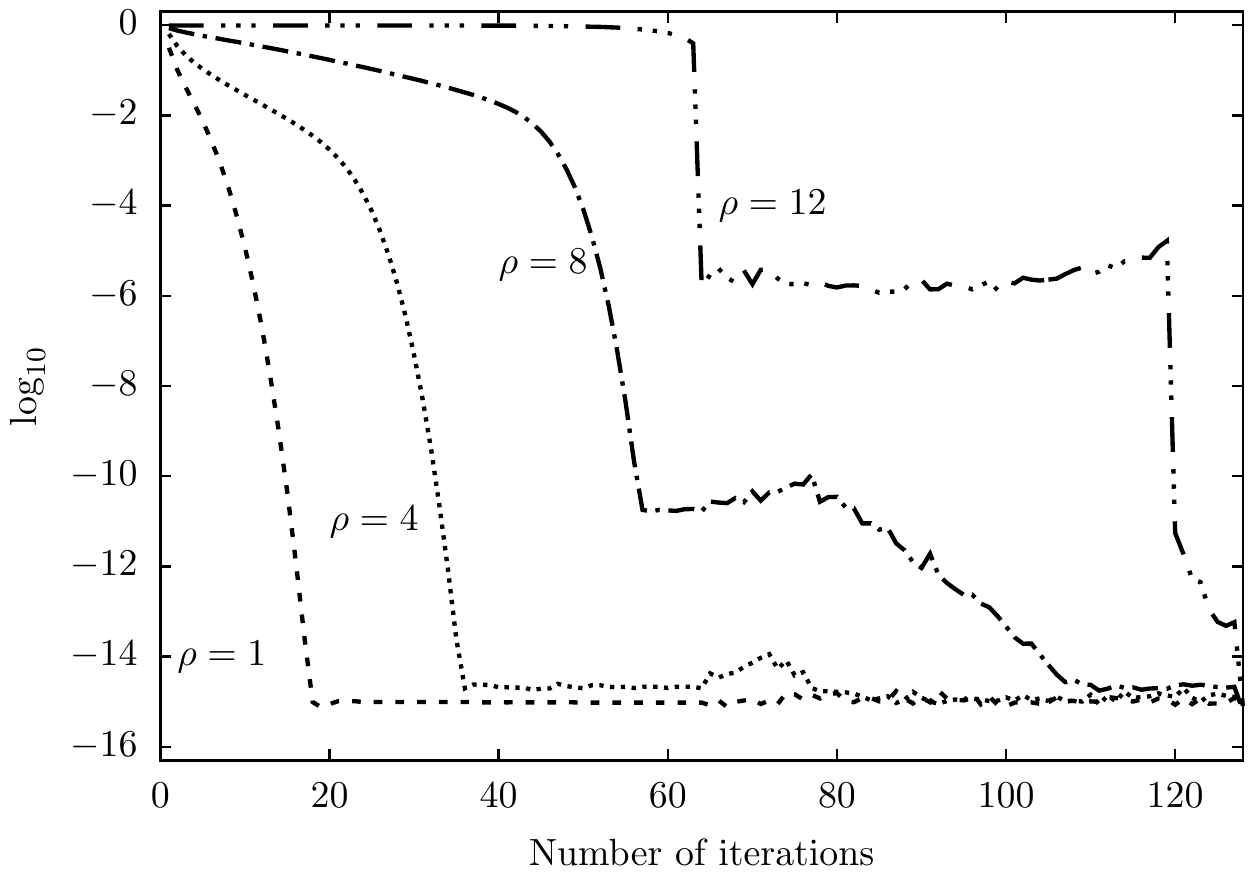}}
	\hspace{6pt}
	\subfloat[Smallest singular value of $H_{k+1, k}$.]{\includegraphics[width=180pt]{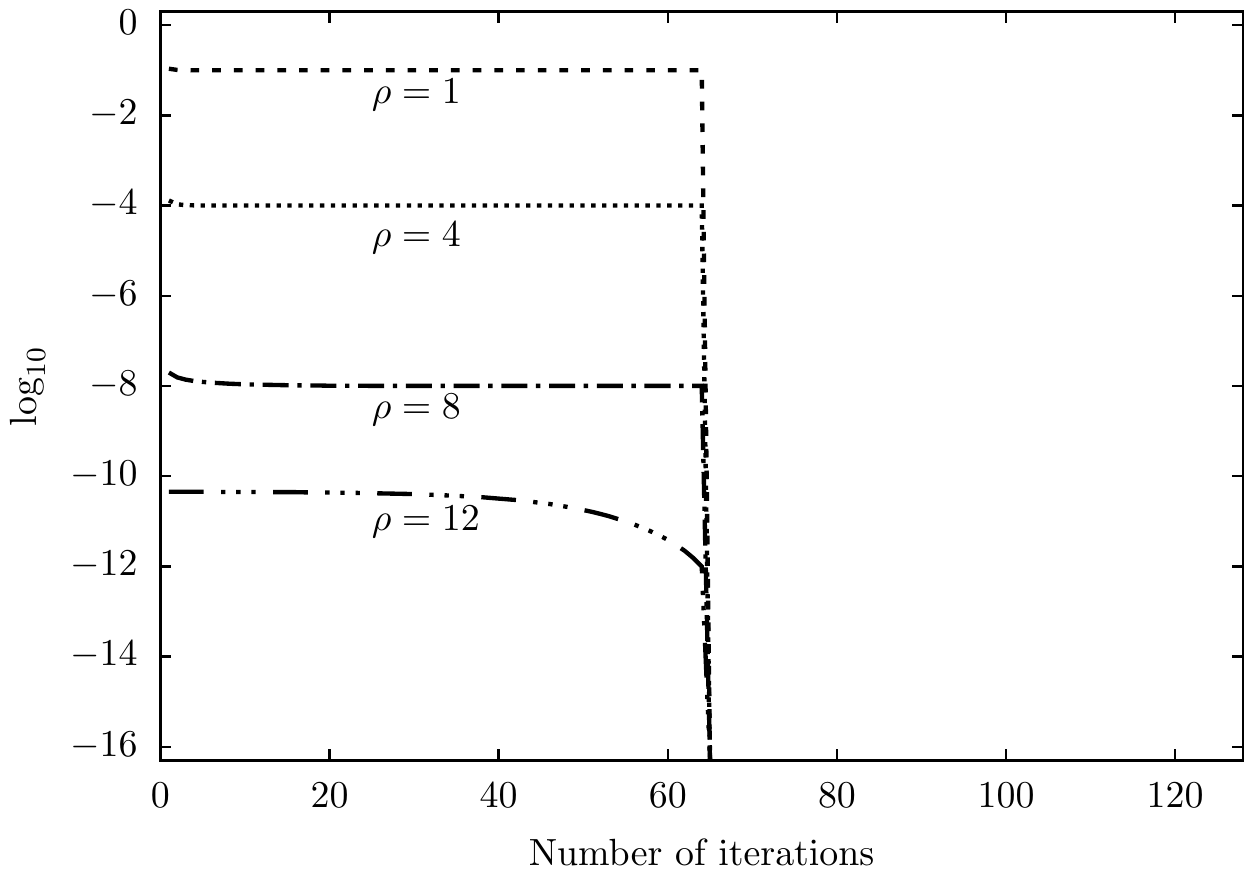}}
	\caption{GMRES on \eqref{eq:DI} with $\boldsymbol{b} = [\boldsymbol{f}^\mathsf{T}, \boldsymbol{0}^\mathsf{T}]^\mathsf{T}$ and different $\kappa (V_1^\mathsf{T} U_1) \simeq 10^\rho / \sqrt{2}$.}
	\label{fig:DI_nrm_r_minsv_b=A[0d]}
	\centering
	\subfloat[Relative residual norm $\| \boldsymbol{r}_k^{\mathrm{R}} \| / \| \boldsymbol{b} \|$.]{\includegraphics[width=180pt]{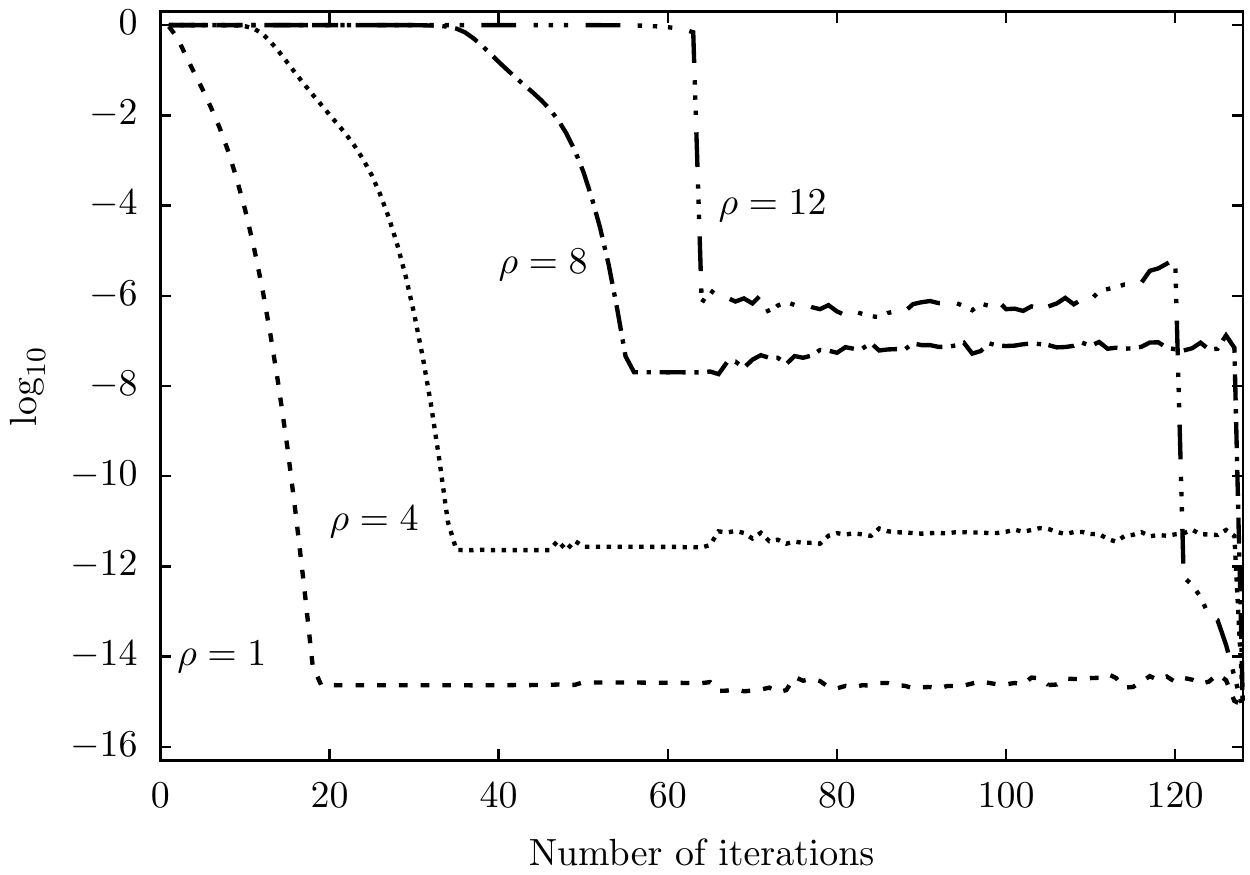}}
	\hspace{6pt}	
	\subfloat[Smallest singular value of $H_{k+1, k}^\mathrm{R}$.]{\includegraphics[width=180pt]{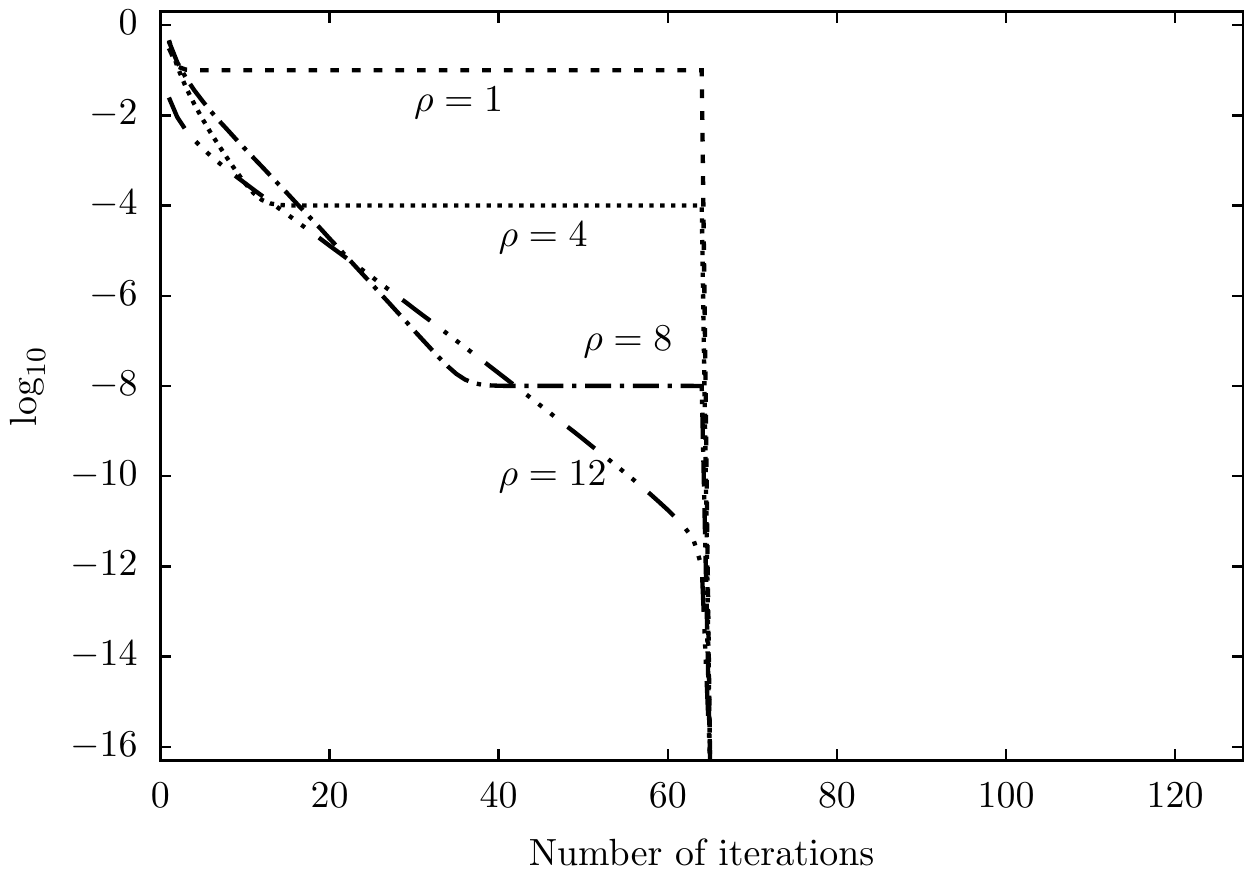}}	
	\caption{RR-GMRES on \eqref{eq:DI} with $\boldsymbol{b} = [\boldsymbol{f}^\mathsf{T}, \boldsymbol{0}^\mathsf{T}]^\mathsf{T}$ and different $\kappa (V_1^\mathsf{T} U_1) \simeq 10^\rho / \sqrt{2}$.}
	\label{fig:DI_RRnrm_r_minsv_b=A[0d]}
\end{figure}

\section{Conclusions} \label{sec:conc}
In this paper we have considered the behavior of the GMRES method for solving a linear system $A \boldsymbol{x} = \boldsymbol{b}$, where $A$ is singular.
We have discussed two classes of singular matrices (EP and GP) satisfying the conditions under which GMRES converges to a least squares solution and to the group inverse solution, respectively.
We have distinguished between the consistent and inconsistent cases and showed that the conditioning of the extended Hessenberg least squares problem can significantly affect the accuracy of approximate solutions computed by GMRES in finite precision arithmetic.

It appears that the consistent EP case is similar to the nonsingular case.
The rank deficiency of the extended Hessenberg least squares problem does not occur and GMRES converges to the accurate approximate solution and terminates with a degeneracy of the Krylov space in the next step.
If the coefficient matrix is EP, but system $A \boldsymbol{x} = \boldsymbol{b}$ is inconsistent, then despite the theoretical guarantee for convergence to the least squares solution, the extended Hessenberg least squares problem becomes seriously ill-conditioned and this may lead to very inaccurate approximate solutions computed by GMRES.
This happens when the distance of the initial residual to the nullspace is too small or when the residual vector converges gradually to the least squares residual.
For such cases, a remedy is to use RR-GMRES.

It is known that if the coefficient matrix is GP and the system is consistent, then theoretically GMRES converges to the solution given by the group inverse of $A$.
We have shown, however, that the extended Hessenberg least squares problem can be ill-conditioned  even in the consistent case.
Indeed, the conditioning of the extended Hessenberg matrix $H_{k+1,k}$ in GMRES depends not only on the conditioning of the coefficient matrix $A$ (as it is in the case of consistent EP problems) but also on the smallest principal angle between the spaces $\mathcal{R}(A^\mathsf{T})$ and $\mathcal{R}(A)$ that can be quite large. In such cases, both GMRES and RR-GMRES may compute inaccurate approximate solutions. 

We believe that under conditions guaranteeing the convergence of GMRES to the generalized least squares solution considered in \cite{WeiWu2000}, our results
can be extended to singular systems with a general $\mathrm{index}(A)$.
Note also that in this paper we assume only exact arithmetic and our considerations form a groundwork for  future work on rounding error analysis.

We would like to point out that here we have considered the behavior of GMRES applied to exactly singular problems.
The extension of our results to GMRES applied to almost singular (or numerically singular) linear systems is far from straightforward as it was also illustrated by our small examples.

\section*{Acknowledgments}
We would like to thank the referees for their valuable comments.

\bibliographystyle{siam}
\bibliography{ref}

\end{document}